\documentclass{article}
\usepackage[utf8]{inputenc}
\usepackage{amsmath}
\usepackage{dsfont}
\usepackage[ruled,vlined]{algorithm2e}
\usepackage{graphicx}
\usepackage{hyperref}
\usepackage{algpseudocode}
\usepackage{bm}
\usepackage[margin = 1in]{geometry}
\usepackage{xcolor}

\DeclareMathOperator{\sign}{sign}

\usepackage{amsmath,amssymb,amsthm}
\newtheorem{theorem}{Theorem}

\newtheorem{definition}{Definition}
\newtheorem{observation}{Observation}

\def\b{\ensuremath\mathbf}

\title{Solving nonlinear ordinary differential equations using the invariant manifolds and Koopman eigenfunctions}

\author{Megan Morrison$^*$ and J. Nathan Kutz$^\dag$\\[.1in]
$^*$ Courant Institute of Mathematical Sciences, New York University, New York, NY 10012\\
$^\dag$ Department of Applied Mathematics, University of Washington, Seattle, WA 98195-3925}
\date{March 2022}

\begin{document}

\maketitle
\begin{abstract}
   Nonlinear ordinary differential equations can rarely be solved  analytically. Koopman operator theory provides a way to solve nonlinear systems by mapping nonlinear dynamics to a linear space using eigenfunctions. Unfortunately, finding such eigenfunctions is difficult. We introduce a method for constructing eigenfunctions from a nonlinear ODE's invariant manifolds.  This method, when successful, allows us to find analytical solutions for constant coefficient nonlinear systems.
   Previous data-driven methods have used Koopman theory to construct local Koopman eigenfunction approximations valid in different regions of phase space; our method finds analytic Koopman eigenfunctions that are exact and globally valid. 
   We demonstrate our Koopman method of solving nonlinear systems on 1-dimensional and 2-dimensional ODEs. The nonlinear examples considered have simple expressions for their invariant manifolds which produce tractable analytical solutions.
   Thus our method allows for the construction of analytical solutions for previously unsolved ordinary differential equations.  It also highlights the connection between invariant manifolds and eigenfunctions in nonlinear ordinary differential equations and presents avenues for extending this method to solve more nonlinear systems.
\end{abstract}


\section{Introduction}

Aside from a few special cases, there are no general methods for solving nonlinear ordinary differential equations, thus necessitating numerical solution techniques  \cite{kutz_data-driven_2013, vidyasagar_nonlinear_2002}.
In contrast, autonomous linear ordinary differential equations are very well understood, with analytical solutions easily constructed from exponential solution forms, i.e. eigen-decompositions~ \cite{boyce2021elementary,strogatz_nonlinear_2016, sontag_mathematical_2013}.
Most of the analysis, prediction, and control of nonlinear systems, in fact, relies on linearization around the fixed points of a given system~\cite{boyce2021elementary} and numerical methods \cite{strogatz_nonlinear_2016, kutz_data-driven_2013, vidyasagar_nonlinear_2002, morrison_nonlinear_2021-1}. 
%
Indeed, the approximation of nonlinear systems in terms of local linear systems is one of the few general methods available for characterizing nonlinear systems.
However, linear analysis and numerical methods for nonlinear systems have limits in terms of their ability to predict and control dynamics.  This motivates our introduction of a principled approach for the construction of analytically tractable invariant manifolds which can be used for expressing solutions of nonlinear differential equations. 

In 1931 B.O. Koopman found that nonlinear Hamiltonian systems could be mapped to an infinite dimensional space of observables with linear dynamics \cite{koopman_hamiltonian_1931, koopman_dynamical_1932}. 
%
Thus, if one finds such a mapping, the dynamical system can be solved in the linear space and the solution can be mapped back to the original nonlinear space.
In some instances, a finite number of Koopman eigenfunctions will form a closed, Koopman-invariant subspace, resulting in a finite linear model which can be easily solved, unlike infinite-dimensional linear models \cite{brunton_koopman_2016}.
The existence and uniqueness of global Koopman eigenfunctions has been proven for stable fixed points and periodic orbits \cite{kvalheim_existence_2021}.
Koopman operator theory provides a way to find analytical solutions of previously unsolved nonlinear systems~\cite{mezic_spectral_2005}, yet finding the eigenfunctions necessary to construct solutions can be very difficult \cite{brunton_koopman_2016, bollt_matching_2018}.
When finding exact analytical solutions is impossible, Koopman operator theory still provides a way to analyze, predict, and control nonlinear systems better than what is possible using linear analysis \cite{brunton_koopman_2016, kaiser_data-driven_2021, korda_optimal_2020}.
Data-driven methods such as {\em dynamic mode decomposition} DMD~\cite{schmid2010dynamic,rowley2009spectral,kutz_dynamic_2016,askham2018variable} and {\em extended dynamic mode decomposition} EDMD\cite{williams_data-driven_2015, williams_kernel-based_2015, kutz_dynamic_2016, arbabi_ergodic_2017} have provided a method for finding approximations to eigenfunctions that map nonlinear dynamics to linear dynamics. These mappings have been especially useful for applications in control \cite{kaiser_data-driven_2021, korda_optimal_2020, proctor_generalizing_2018}.
Other approaches have divided the domain of nonlinear ODEs into separate ``basins of attraction" where separate linearization transforms are constructed for different regions \cite{lan_linearization_2013}.
Koopman methods can also be useful for performing phase-amplitude reductions for nonlinear systems containing limit cycle oscillators \cite{wilson_isostable_2019, wilson_phase-amplitude_2020}.

In this work, we develop a principled method for finding sets of eigenfunctions that map nonlinear dynamics to finite linear systems; these eigenfunctions allow us to construct solutions to nonlinear ordinary differential equations. Unlike the data-driven approaches which create approximate eigenfunctions by extending the linearization around fixed points, we construct exact, global eigenfunctions from invariant manifolds in the system. 
The manifolds have simple analytic expressions which allow us to construct analytical solutions
to the nonlinear system from these eigenfunctions. 
%
%
Our method is shown in 
Fig.~\ref{fig:intro_fig} where a nonlinear ODE has eigenfunctions that can be constructed from the system's invariant manifolds. The nonlinear system has three invariant manifolds (red lines, Fig~\ref{fig:intro_fig}(a)) which are used to construct two globally valid eigenfunctions $\varphi_1(x,y)$ and $\varphi_2(x,y)$, both which have linear dynamics (Fig.~\ref{fig:intro_fig}(b)). The solutions in the linear $(\varphi_1, \varphi_2)$ space are mapped back to the original $(x,y)$ space, resulting in exact analytical solutions for the nonlinear ODE (Fig.~\ref{fig:intro_fig}(c)).  This example highlights how a principled algorithmic approach can be used to determine solutions to nonlinear systems of differential equations.

The manuscript is organized as follows.
Section~\ref{sec:background} provides an overview of Koopman theory, eigenfunctions, and the invariant manifolds of nonlinear ODEs. Section~\ref{sec:construct_eigs} states the conditions under which eigenfunctions can be constructed from invariant manifolds and formulas for the eigenvalue-eigenfunction pairs. This section also considers the conditions necessary to obtain analytical solutions from sets of eigenfunctions.
Section~\ref{sec:1D_examples} outlines the Koopman eigenfunction approach to solving 1-dimensional ODEs and provides several examples. Section~\ref{sec:2D_examples} outlines the Koopman eigenfunction approach to solving 2-dimension ODEs and demonstrates the method with multiple examples. Section~\ref{sec:discussion} discusses the method's limitations, possible extensions, and ramifications for data-driven discovery of eigenfunctions. Section~\ref{sec:conclusion} concludes the manuscript.

\begin{figure}[t]
    \centering
    \includegraphics[width=0.9\linewidth]{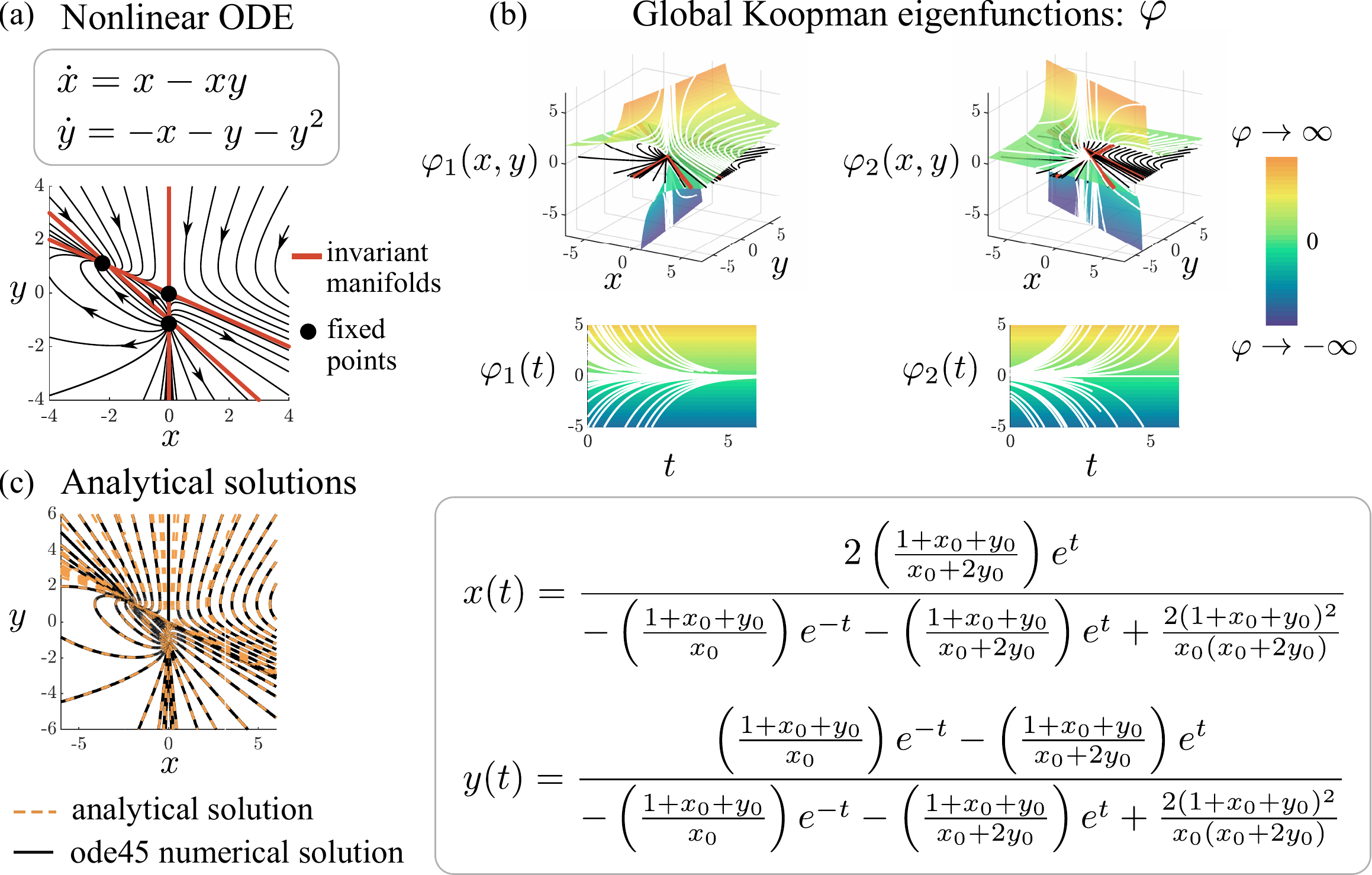}
    \caption{(a) Nonlinear ODE with multiple fixed points and real invariant manifolds (red) (b) Global Koopman eigenfunctions have linear dynamics.  (c) Analytical solution constructed from the global Koopman eigenfunctions}
    \label{fig:intro_fig}
\end{figure}

\section{Background}\label{sec:background}


We consider how to construct solutions to nonlinear, autonomous, ordinary differential equations by using Koopman theory to map nonlinear dynamics to a space that has linear dynamics. We produce the nonlinear-to-linear mapping with Koopman eigenfunctions that we construct from the nonlinear ODE's invariant manifolds.

\subsection{Koopman theory}

Consider the ordinary differential equation
\begin{align}
    \frac{d\b{x}}{dt} = F(\b{x}), \quad \b{x} \in \mathds{R}^n, \label{eq:general_ode}
\end{align}
with the autonomous vector field $F: \mathds{R}^n \rightarrow \mathds{R}^n$ that operates on the state vector $\b{x}$.
The flow associated with Eq.~\ref{eq:general_ode} for each $t \in \mathds{R}$ is the function $\b{x}(t) := S_t(\b{x}_0)$ for a trajectory starting at $\b{x}(0) = \b{x}_0 \in \mathds{R}^n$.
The Koopman operator describes the dynamics of ``observables" or measurements of the state vector along its flow \cite{koopman_hamiltonian_1931,mezic_spectral_2005, budisic_applied_2012, bollt_geometric_2021}. The observable measurements $g: \mathds{R}^n \rightarrow \mathds{C}$ are elements of a space of observable functions $\mathcal{F}$.
The Koopman operator $\mathcal{K}_t: \mathcal{F} \rightarrow \mathcal{F}$ is an infinite-dimensional linear operator that propogates observables $g$ of the state vector $\b{x}$ forward in time along trajectories of Eq.~\ref{eq:general_ode},
\begin{align}
    \mathcal{K}_t [g](\b{x}) = g \circ S_t(\b{x}). \label{eq:koopman}
\end{align}
The left-hand side of Eq.~\ref{eq:koopman} states that the Koopman operator $\mathcal{K}_t$ pushes the observables $g$ of the state vector $\b{x}$ forward in time $t$. The resulting value is equivalent to the right-hand side of the equation which states that the original variable $\b{x}$ is pushed forward in time  $t$ according to flow $S_t$ of Eq.~\ref{eq:general_ode} and then observed by $g$. Figure~\ref{fig:into_Koopman}(a) illustrates the Koopman operator pushing observables $g$ forward in time and how the result is equivalent to taking measurements of $\b{x}$ along its flow.

\begin{figure}[t]
    \centering
    \includegraphics[width=0.9\linewidth]{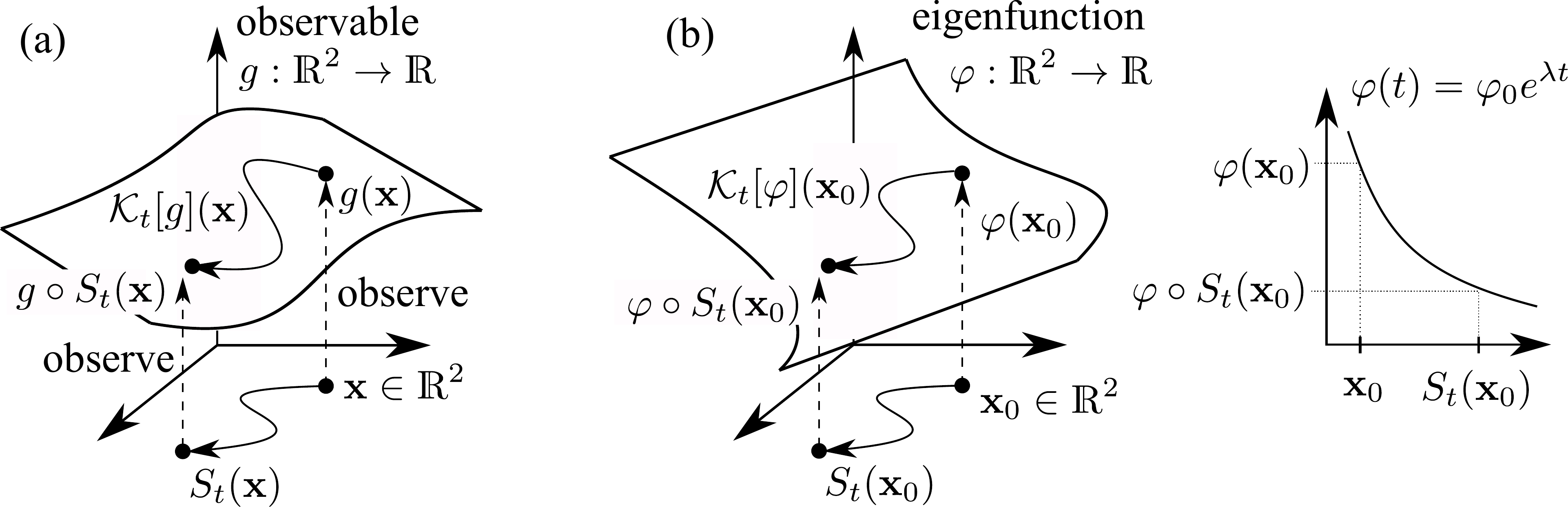}
    \caption{(a) The Koopman operator $\mathcal{K}_t$ pushes  observables $g$ forward in time.  (b) Eigenfunctions $\varphi$ are special observables with linear dynamics, $\dot{\varphi} = \lambda \varphi$.}
    \label{fig:into_Koopman}
\end{figure}

Eigenfunctions $\varphi \in \mathcal{F}$ of the Koopman operator are special observables that have exponential time dependence \cite{bollt_geometric_2021}.  An eigenvalue-eigenfunction pair ($\lambda, \varphi$) of $\mathcal{K}_t$ satisfies the equation
\begin{align}
    \mathcal{K}_t[\varphi](\b{x}) = \varphi(\b{x}) e^{\lambda t}.
\end{align}
$\varphi(t)$ grows exponentially with eigenvalue (growth constant) $\lambda$ and initial condition $\varphi(\b{x})$ under the operator $\mathcal{K}_t$.
Eigenfunctions have linear dynamics and satisfy the following equations:
\begin{align}
    \frac{d}{dt}\varphi(\b{x}) &= \lambda \varphi(\b{x}) \label{eq:lin_dyn_varphi}\\
        \nabla_{\b{x}} \varphi(\b{x})  \cdot  \frac{d \b{x}}{dt} &= \lambda \varphi(\b{x}) \nonumber \\
    \nabla_{\b{x}} \varphi(\b{x})  \cdot  F(\b{x}) &= \lambda \varphi(\b{x}). \label{eq:lin_dyn_varphi_pde}
\end{align}
Solutions to this linear first-order partial differential equation are eigenfunctions of $\mathcal{K}_t$ \cite{bollt_geometric_2021, bollt_matching_2018}. Figure~\ref{fig:into_Koopman}(b) illustrates that eigenfunctions under the Koopman operator have linear dynamics with the closed-form solution 
\begin{align}
    \varphi(t,\b{x}_0) = \varphi(\b{x}_0)e^{\lambda t}. \label{eq:varphi_t}
\end{align}
Eigenfunctions are extremely useful as they can be used to construct analytical solutions for nonlinear ODEs.
There is no known method for finding analytical solutions for most nonlinear ODEs; finding eigenfunctions for such ODEs is a promising method for obtaining solutions. 
Unfortunately, despite the growing interest and usefulness of Koopman eigenfunctions, there are few methods available to discover explicit, closed-form expressions for Koopman eigenfunctions \cite{bollt_matching_2018,kaiser_data-driven_2021,bollt_geometric_2021,page_koopman_2019}.

We show that, when certain conditions are met, closed form expressions for Koopman eigenfunctions can be constructed from invariant manifold generating functions. These are the first examples of combining multiple invariant manifolds to construct eigenfunctions that are rational expressions. These are also among the first examples of finding eigenfunctions for nonlinear ODEs that have multiple fixed points and that are not Hamiltonian.

Another way in which we deviate from previous methods is that upon finding eigenfunctions, we do not construct solutions for $\b{x}(t)$ by using a linear combination of eigenfunctions. Rather the closed-form solution for $\b{x}(t)$ is constructed using a \textit{nonlinear} combination of eigenfunctions.
This method allows for closed-form analytical solutions for $\b{x}(t)$ that were previously not possible.

\subsection{Invariant manifolds of ordinary differential equations}

Our method for obtaining eigenfunctions for 2D nonlinear ODEs requires finding invariant manifolds.
An invariant manifold of a 2-dimensional ODE, $\dot{\b{x}} = F(\b{x})$, $\b{x} \in \mathds{R}^2$, can be defined by an implicit function $\Lambda = \{(x,y): M(x,y)=0 \}$. 
The invariant manifold $\Lambda$ is defined using the invariant manifold generating function $M(x,y)$ which has a zero level-set along the invariant manifold. 
The invariant manifold generating functions ($M$-functions) of a 2-dimensional ODE are the building blocks for constructing eigenfunctions.
An invariant manifold of a dynamical system is a topological manifold that is invariant under the actions of the dynamical system \cite{chicone_ordinary_2006,wiggins_introduction_2003,vidyasagar_nonlinear_2002}.

\begin{definition}[Invariant set]
  A set of states $\Lambda \subseteq \mathds{R}^n$ of Eq.~\ref{eq:general_ode} is called an invariant set of Eq.~\ref{eq:general_ode} if for all $\b{x}_0 \in \Lambda$, and for all $t>0$, $\b{x}(t) \in \Lambda$.
\end{definition}

Consider an invariant set defined by the curve $M(x,y)=0$. The flow $\b{x}(t) = S_t(\b{x}_0)$ will stay on the curve $M(x,y)=0$ for all time if the initial condition $\b{x}_0$ is on the curve.
Invariant sets are often defined as emanating from equilibrium points. An equilibrium point $\mathcal{P} \in \mathds{R}^n$ of Eq.~\ref{eq:general_ode} is a point such that $F(\mathcal{P}) = 0$. A stable manifold of $\mathcal{P}$ is the set of points that is attracted to $\mathcal{P}$ forward in time while an unstable manifold of $\mathcal{P}$ is the set of points that is attracted to $\mathcal{P}$ backward in time \cite{wiggins_introduction_2003, homburg_invariant_2006,palis_hyperbolicity_1995}.
\begin{definition}[Stable manifold of an equilibrium point]
    The stable manifold of an equilibrium point $\mathcal{P}$ of Eq.~\ref{eq:general_ode} is defined as $$W^S(\mathcal{P}) = \{ \b{x}: S_t(\b{x}) \rightarrow \mathcal{P} \ \ as \ \ t \rightarrow \infty \}$$
\end{definition}
\begin{definition}[Unstable manifold of an equilibrium point]
    The unstable manifold of an equilibrium point $\mathcal{P}$ of Eq.~\ref{eq:general_ode} is defined as $$W^U(\mathcal{P}) = \{ \b{x}: S_t(\b{x}) \rightarrow \mathcal{P} \ \ as \ \ t \rightarrow -\infty \}$$
\end{definition}
The eigenvectors of a system linearized around a fixed point are tangent to the system's stable and unstable manifolds \cite{wiggins_introduction_2003, homburg_computation_1995}; $E^S(\mathcal{P})$ is tangent to $W^S(\mathcal{P})$ and  $E^U(\mathcal{P})$ is tangent to $W^U(\mathcal{P})$ at $\mathcal{P}$. For 2-dimensional nonlinear ODEs, we can use the fixed points $\mathcal{P}$ and the eigenvector directions, $E^S(\mathcal{P})$ and $E^U(\mathcal{P})$, to begin solving for the closed form expressions, $M(x,y) = 0$, that describe the invariant manifolds of the system \cite{roberts_utility_1989, homburg_computation_1995}.
Figure~\ref{fig:into_manifold} illustrates invariant manifolds emanating from a saddle fixed point of a system in $\mathds{R}^2$. The invariant manifold $\Lambda_1$ is the set of points $\b{x}$ that satisfy $M_1(\b{x})=0$ (Fig.~\ref{fig:into_manifold}(a)).
$\Lambda_1$ is a stable manifold of $\mathcal{P}$ since $\forall \b{x} \in \Lambda_1$  $\lim_{t \rightarrow \infty} S_t(\b{x}) = \mathcal{P}$. The invariant manifold $\Lambda_2$ is the set of points $\b{x}$ that satisfy $M_2(\b{x})=0$. $\Lambda_2$ is an unstable manifold of $\mathcal{P}$ since $\forall \b{x} \in \Lambda_2$  $\lim_{t \rightarrow -\infty} S_t(\b{x}) = \mathcal{P}$. Figure~\ref{fig:into_manifold}(b) illustrates how eigenvectors $E^S(\mathcal{P})$ and $E^U(\mathcal{P})$ of the system linearized at equilibrium point $\mathcal{P}$ are tangent to the stable and unstable manifolds at $\mathcal{P}$. This information is useful when solving for the manifold generating functions $M(\b{x})$ whose zero level-set curves implicitly define the invariant manifolds.

\begin{figure}[t]
    \centering
    \includegraphics[width=0.9\linewidth]{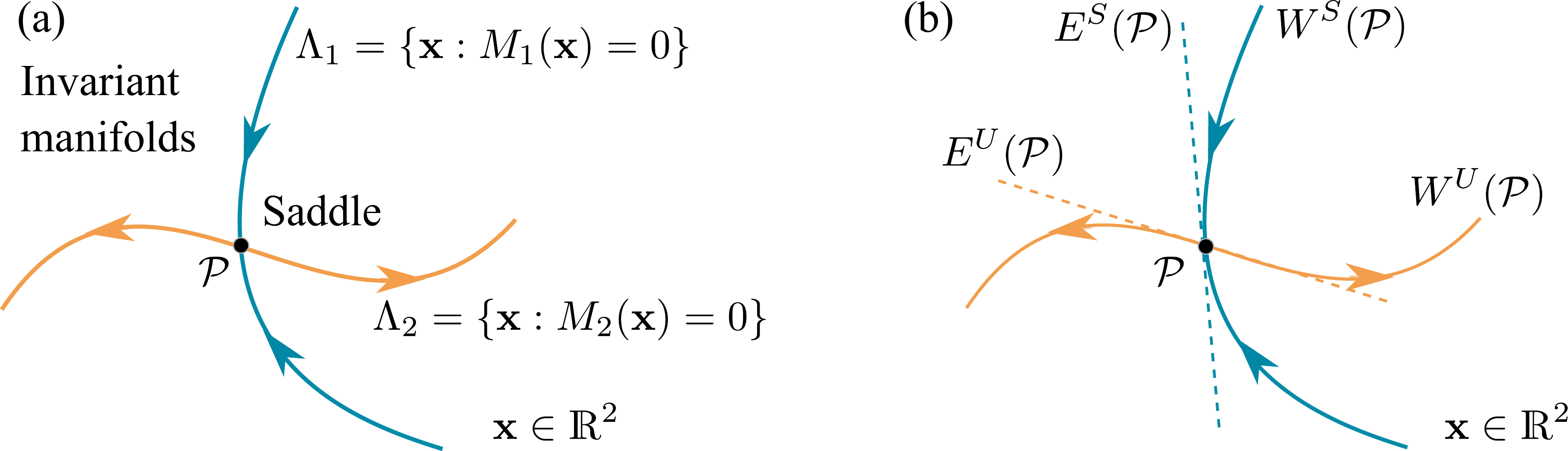}
    \caption{(a) Invariant sets $\Lambda_1$ and $\Lambda_2$ defined by curves $M_1(\b{x})=0$ and $M_2(\b{x})=0$ emanate from the saddle fixed point $\mathcal{P}$. (b) $\Lambda_1=W^S(\mathcal{P})$ is a stable manifold of $\mathcal{P}$, and $\Lambda_2=W^U(\mathcal{P})$ is an unstable manifold of $\mathcal{P}$. Eigenvector directions $E^S(\mathcal{P})$ and  $E^U(\mathcal{P})$ are tangent to the stable and unstable manifolds.}
    \label{fig:into_manifold}
\end{figure}

The $M$-functions of a 2-dimensional ODE have corresponding $N$-functions that tell us how the $M$-functions must be combined in order to generate eigenfunctions for the system. We solve for these $N$-functions by differentiating $M(\b{x})$.
The level-set $M(\b{x}) = 0$ of the surface $M(\b{x})$ is an invariant manifold of the differential equation $\dot{\b{x}} = F(\b{x})$, $\b{x} \in \mathds{R}^2$, if
$\frac{d}{dt}M(\b{x}) = 0$ along the level set $M(\b{x}) = 0$.  Therefore, $M(\b{x}) = 0$ is an invariant manifold of $\dot{\b{x}} = F(\b{x})$ if
\begin{align}
    \frac{d}{dt}M(\b{x}) = M(\b{x})N(\b{x}) \label{eq:invar_man_req}
\end{align}
for some $N(\b{x})$. $N(\b{x})$ is the $N$-function of $M(\b{x})$.
The $M$-functions and their corresponding $N$-functions have the following interesting properties that aid us in constructing eigenfunctions.

\begin{observation}
If $\frac{d}{dt}M_1(\b{x}) = M_1(\b{x}) N_1(\b{x})$ and $\frac{d}{dt}M_2(\b{x}) = M_2(\b{x}) N_2(\b{x})$
   then $\frac{d}{dt}(M_1(\b{x})M_2(\b{x})) = (N_1(\b{x})+N_2(\b{x}))M_1(\b{x})M_2(\b{x})$.
\end{observation}

\begin{observation}
    If $\frac{d}{dt}M_1(\b{x}) = M_1(\b{x}) N_1(\b{x})$ and $\frac{d}{dt}M_2(\b{x}) = M_2(\b{x}) N_2(\b{x})$
   then $\frac{d}{dt}\left(\frac{M_1(\b{x})}{M_2(\b{x})} \right) = (N_1(\b{x})-N_2(\b{x}))\left(\frac{M_1(\b{x})}{M_2(\b{x})} \right)$.
\end{observation}
In the next section we will use $M$-functions and $N$-functions to solve for eigenvalue-eigenfunction pairs of the Koopman operator $\mathcal{K}_t$.
We construct eigenvalues from the $N$-functions and eigenfunctions from the $M$-functions.

\section{Constructing eigenfunctions from invariant manifolds}\label{sec:construct_eigs}

Notice that if $N(\b{x})=c \in \mathds{R}$, a constant, then $M(\b{x})$ is an eigenfunction of $\dot{\b{x}}=F(\b{x})$. If $N(\b{x})$ is not a constant, then $M(\b{x})$ cannot be an eigenfunction.
Nonetheless, a combination of manifold generating functions, $M_1(\b{x})$ and $M_2(\b{x})$, may still create an eigenfunction so long as their corresponding $N$-functions can linearly combine to produce a constant.
Theorem~\ref{thm:manifolds_general} allows us to construct eigenfunctions given there exists some linear combination of $N$-functions that results in a constant.

\begin{theorem} \label{thm:manifolds_general}
Let $M_1(\b{x})$ and $M_2(\b{x})$ be invariant manifold generating functions of $\dot{\b{x}} = F(\b{x})$, $\b{x} \in \mathds{R}^n$, such that $\frac{d}{dt}M_1(\b{x}) = M_1(\b{x})N_1(\b{x})$ and $\frac{d}{dt}M_2(\b{x}) = M_2(\b{x})N_2(\b{x})$ for some functions $N_1(\b{x})$ and $N_2(\b{x})$. $\varphi(\b{x}) = M_1^p(\b{x})M_2^q(\b{x})$, $p,q \in \mathds{C}$, is an eigenfunction of $\dot{\b{x}} = F(\b{x})$ if and only if $p N_1(\b{x})+ q N_2(\b{x}) = \lambda$ for some constant $\lambda \in \mathds{C}$.
\end{theorem}

\begin{proof}
Let $\frac{d}{dt}M_1(\b{x}) = M_1(\b{x})N_1(\b{x})$ and $\frac{d}{dt}M_2(\b{x}) = M_2(\b{x})N_2(\b{x})$ for $\dot{\b{x}} = F(\b{x})$, $\b{x} \in \mathds{R}^n$. Let $\varphi(\b{x}) =M_1^p(\b{x})M_2^q(\b{x})$ and $\lambda = p N_1(\b{x})+q N_2(\b{x})$ for some constant $\lambda \in \mathds{C}$. 
\begin{align*}
    \frac{d}{dt} \varphi(\b{x}) =& \frac{d}{dt}\left[ M_1^p(\b{x})M_2^q(\b{x}) \right]\\
    =& M_1^p(\b{x}) \frac{d}{dt} M_2^q(\b{x}) + M_2^q(\b{x}) \frac{d}{dt}M_1^p(\b{x})\\
    =& M_1^p q M_2^{q-1}\frac{d}{dt} M_2 + M_2^qp M_1^{p-1} \frac{d}{dt} M_1 \\
    =& M_1^p q M_2^{q-1} M_2 N_2 + M_2^qp M_1^{p-1} M_1 N_1\\
    =& M_1^p q M_2^{q} N_2 + M_2^q p M_1^{p} N_1\\
    =& (p N_1 + q N_2)M_1^p M_2^q\\
    =& \lambda \varphi(\b{x})
\end{align*}
Therefore, $\varphi(\b{x}) = M_1^p(\b{x})M_2^q(\b{x})$, $p,q \in \mathds{C}$, is an eigenfunction of $\dot{\b{x}} = F(\b{x})$ with corresponding eigenvalue $\lambda = p N_1(\b{x})+ q N_2(\b{x})$.
\end{proof}

Theorem~\ref{thm:manifolds_general} says that if there exists a linear combination of $N_1(\b{x})$ and $N_2(\b{x})$ that results in a constant $\lambda$, then we can construct eigenfunctions from the corresponding $M$-functions. $\lambda$ is the eigenvalue that corresponds to the resulting eigenfunction, producing the eigenvalue/eigenfunction pair $(\lambda, \varphi(\b{x}))$. We note that if there exists a $p$ and $q$ such that $p N_1(\b{x}) + q N_2(\b{x}) = \lambda$, where $\lambda$ is a constant, then $c(p N_1(\b{x}) + q N_2(\b{x})) = c \lambda$, where $c \lambda$ is also a constant for any $c \in \mathds{C}$. This tells us that from the single eigenvalue/eigenfunction pair $(\lambda, \varphi(\b{x}))$, we can generate a family of eigenvalue/eigenfunction pairs $(c \lambda, \varphi^c(\b{x}))$, where $c \in \mathds{C}.$ We observe from this eigenvalue/eigenfunction family that we may set any complex number to be the eigenvalue; however, the resulting corresponding eigenfunction may be quite complicated. Therefore, in the following examples we choose eigenvalues that will result in eigenfunctions that have simple expressions.

Theorem~\ref{thm:manifolds_general} can be extended to include more invariant manifolds than simply two. Nonlinear systems may contain more than two invariant manifolds and a linear combination of $N$-functions from more than two of these invariant manifolds may be required to produce a constant.

\begin{theorem} \label{thm:manifolds_general2}
Let $\{M_i(\b{x}) \}_i^k$ be invariant manifold generating functions of $\dot{\b{x}} = F(\b{x})$, $\b{x} \in \mathds{R}^n$, such that $\forall i$, $\frac{d}{dt}M_i(\b{x}) = M_i(\b{x})N_i(\b{x})$ for some functions $N_i(\b{x})$. $\varphi(\b{x}) = \prod_i^k M_i^{p_i}(\b{x})$, $p_i \in \mathds{C}$, is an eigenfunction of $\dot{\b{x}} = F(\b{x})$ if and only if $\sum_i^k p_i N_i(\b{x})= \lambda$ for some constant $\lambda \in \mathds{C}$.
\end{theorem}

\begin{proof}
Let $\frac{d}{dt}M_i(\b{x}) = M_i(\b{x})N_i(\b{x})$, $i \in \{1,...,k \}$, for $\dot{\b{x}} = F(\b{x})$, $\b{x} \in \mathds{R}^n$. Let $\varphi(\b{x}) =\prod_i^k M_i^{p_i}(\b{x})$ and $\lambda = \sum_i^k p_i N_i(\b{x})$ for some constant $\lambda \in \mathds{C}$. 

\begin{align*}
    \frac{d}{dt} \varphi(\b{x}) =& \frac{d}{dt}\left(\prod_{i=1}^k M_i^{p_i}(\b{x}) \right)\\
    =& \frac{d}{dt} M_1^{p_1} \left(\prod_{i=2}^k M_i^{p_i} \right)  + M_1^{p_1} \frac{d}{dt}\left(\prod_{i=2}^k M_i^{p_i} \right) \\
    =& p_1 M_1^{p_1} N_1 \left(\prod_{i=2}^k M_i^{p_i} \right)  + M_1^{p_1} \frac{d}{dt}\left(\prod_{i=2}^k M_i^{p_i} \right)\\
    =& p_1 N_1 \left(\prod_{i=1}^k M_i^{p_i} \right)  + M_1^{p_1} \frac{d}{dt}\left(\prod_{i=2}^k M_i^{p_i} \right)\\
    =& p_1 N_1 \left(\prod_{i=1}^k M_i^{p_i} \right)  + M_1^{p_1} \left( p_2 N_2 \left(\prod_{i=2}^k M_i^{p_i} \right)  + M_2^{p_2} \frac{d}{dt}\left(\prod_{i=3}^k M_i^{p_i} \right) \right)\\
    =& p_1 N_1 \left(\prod_{i=1}^k M_i^{p_i} \right)  +  p_2 N_2 \left(\prod_{i=1}^k M_i^{p_i} \right)  + M_1^{p_1} M_2^{p_2} \frac{d}{dt}\left(\prod_{i=3}^k M_i^{p_i} \right)\\
    =& \left( \sum_{i=1}^k p_i N_i \right) \left(\prod_{i=1}^k M_i^{p_i} \right) \\
    =& \left( \sum_{i=1}^k p_i N_i(\b{x}) \right) \varphi(\b{x})\\
    =& \lambda \varphi(\b{x})
\end{align*}

Therefore, $\varphi(\b{x}) = \prod_i^k M_i^{p_i}(\b{x})$ is an eigenfunction of $\dot{\b{x}} = F(\b{x})$ with corresponding eigenvalue $\lambda = \sum_i^k p_i N_i(\b{x})$.
\end{proof}

We use the previous theorems to construct eigenfunctions for 2-dimensional nonlinear ODEs by (1) finding closed-form expressions for the invariant manifolds, (2) solving for the corresponding $N$-functions, and (3) finding combinations of $N$-functions that reduce to a constant. Once we have solved for the $N$-function weights, $p_i$, we use these weights as exponents in the $M$-function product, producing an eigenfunction for the nonlinear ODE.

\subsection{Obtaining independent eigenfunctions}

To obtain a solution for an n-dimensional ODE, we must construct at least $n$ independent eigenfunctions that have different level sets so that we may have a unique mapping from the eigenfunctions back to the original variables.
To obtain solutions for 2-dimensional ODEs, we must construct at least two independent eigenfunctions, $\varphi_1(\b{x})$ and $\varphi_2(\b{x})$ that have different level sets --- that is, they cannot belong to the same ``family" or equivalence class of eigenfunctions.
An equivalence class of Koopman eigenvalue-eigenfunction pairs, $\overline{(\lambda, \varphi(\b{x}))}$, have level sets that match between functions (although the level sets need not match to the same levels). Exponentiations of an eigenfunction belong to the same equivalence class \cite{bollt_geometric_2021, budisic_applied_2012},
$$ \{(p \lambda, \varphi^p(\b{x})), p \in \mathds{R}  \} \subset \overline{(\lambda, \varphi(\b{x}))}.$$
Multiples of an eigenfunction also belong to the same equivalence class,
$$ \{(\lambda, \alpha \varphi(\b{x})), \alpha \in \mathds{R}  \} \subset \overline{(\lambda, \varphi(\b{x}))}.$$
If one of our eigenfunctions constructed from invariant manifolds is not an exponentiated multiple of the other, then the two eigenfunctions belong to different equivalence classes and we may use the pair of independent eigenfunctions to solve the ODE.
We can use the following theorem to find eigenfunctions that belong to different equivalence classes.

\begin{theorem} \label{thm:lin_independence}
Let $\lambda_1 = \sum_i^k p_i N_i(\b{x})$ and $\lambda_2 = \sum_i^k q_i N_i(\b{x})$ be the eigenvalues corresponding to eigenfunctions $\varphi_1(\b{x}) = \prod_i^k M_i^{p_i}(\b{x})$ and $\varphi_2(\b{x}) = \prod_i^k M_i^{q_i}(\b{x})$ where the weights for the linear combination of $N$-functions are the vectors $\b{p} = [p_1 \ \ p_2 \ \ ... \ \ p_k]^T$ and $\b{q} = [q_1 \ \ q_2 \ \ ... \ \ q_k]^T$.
The eigenfunctions $\varphi_1(\b{x})$ and $\varphi_2(\b{x})$ belong to different equivalence classes if $\b{p}$ and $\b{q}$ are linearly independent vectors, that is $\b{p} \neq \alpha \b{q}$ for any $\alpha \in \mathds{C}.$
\end{theorem}

\begin{proof}
If $\varphi_1(\b{x})$ and $\varphi_2(\b{x})$ are in the same equivalence class then 
$$
\prod_i^k M_i^{p_i}(\b{x}) = \left(\prod_i^k M_i^{q_i}(\b{x}) \right)^\alpha = \prod_i^k M_i^{\alpha q_i}(\b{x}) \ \ \text{for some} \ \ \alpha \in \mathds{R}.
$$
This implies that $p_i = \alpha q_i$ for all $i \in \{1,...,k\}$. However, $\b{p} \neq \alpha \b{q}$ for any $\alpha$. Therefore,
$\varphi_1(\b{x}) \neq (\varphi_1(\b{x}))^\alpha$;
$\varphi_1(\b{x})$ and $\varphi_2(\b{x})$ are not in the same equivalence class.
\end{proof}

Theorem~\ref{thm:lin_independence} says that if the weighting vectors for the $N$-functions are linearly independent then the eigenfunctions constructed from the $M$-functions are independent (in different equivalence classes).


\section{Koopman eigenfunctions for 1-dimensional ODEs}\label{sec:1D_examples}

We first consider how to use a Koopman approach to solve nonlinear, autonomous, first-order ordinary differential equations. 
In the 1-dimensional case, such equations are easily solvable via separation of variables. However, we will consider the alternative, Koopman approach to solving these differential equations in order to build an intuition for the method in the 2-dimensional case, where separation of variables can no longer be used to construct a solution.
Consider a nonlinear, autonomous, first-order, ordinary differential equation
\begin{align}
    \frac{dx}{dt} = f(x), \quad x \in \mathds{C}.
\end{align}
This is a separable, first-order differential equation and so is solvable by separating the variables and then integrating,
\begin{align}
    \int \frac{dx}{f(x)} = \int dt = t+c.
\end{align}
Instead of solving this ODE directly, we can instead take the Koopman perspective and first map the nonlinear dynamics of $x$ to a space with linear dynamics, find a solution in the linear space, and then map the solution back to $x$.
We solve for the eigenfunction $\varphi: \mathds{C} \rightarrow \mathds{C}$ that maps the nonlinear dynamics to a space with linear dynamics,
\begin{align}
    \frac{d}{dt}\varphi(x) &= \frac{d}{dx}\varphi(x) \frac{dx}{dt} = \varphi'(x) f(x) = \lambda \varphi(x) \nonumber \\
    \int \frac{\varphi'(x)}{\varphi(x)}dx &= \int \frac{\lambda}{f(x)} dx \nonumber \\
    \ln[\varphi(x)]+c_1 &= \int \frac{\lambda}{f(x)} dx \nonumber \\
    \varphi(x) &= c_2 e^{\int \frac{\lambda}{f(x)} dx}.
\end{align}
If  $f(x)$ is a polynomial with simple roots, $f(x) = c\prod_{i=1}^n (x-x_i)$, where $x_i \in \mathds{C}$, then we can solve further by integrating each of the resulting fractions separately,
\begin{align*}
    \int \frac{\lambda}{f(x)}dx = \int \frac{\lambda}{c\prod_{i=1}^n (x-x_i)}dx =\int \sum_{i=1}^n \frac{p_i}{x-x_i}dx = \sum_{i=1}^n \log[(x-x_i)^{p_i}].
\end{align*}
The numerators are determined via the method of partial fractions,
\begin{align}
    p_i = \frac{\lambda}{c\prod_{j = 1:n, j \neq i} (x_i - x_j)} \label{eq:pi}.
\end{align}
Therefore the solution to the eigenfunction is
\begin{align}
    \varphi(x) = c_2 e^{\int \frac{\lambda}{f(x)}dx} = c_2 \prod_{i=1}^n (x-x_i)^{p_i}, \label{eq:final_g}
\end{align}
where $\{ x_i  \}_{i}^n$ are the simple roots of $f(x)$ and the exponents $\{ p_i  \}_{i}^n$ are the constants generated by the method of partial fractions (Eq.~\ref{eq:pi}). Notice that $\varphi(x)$ is a composition of zeros and singularities and the sign of $\lambda$ determines which $x_i$ are zeros versus singularities. 
The dynamics of $\varphi$ are linear by construction,
\begin{align}
    \varphi(t,x_0) = \varphi(x_0) e^{\lambda t}
    .
\end{align}
Lastly, we solve for $x$ as a function of $\varphi(t,x_0)$ by inverting $\varphi(x)$ (Eq.~\ref{eq:final_g}). $\varphi(x)$ is often not invertible, meaning that multiple $x$ map to a single $\varphi$ value.  Only by including knowledge of the initial condition $x_0$ can this ambiguity be resolved, allowing us to create a one-to-one mapping from $(\varphi, x_0) \mapsto x$.

\subsection{1-dimensional ODE --- example 1}

Let us first consider a nonlinear differential equation that has been used as an example in previous work on Koopman analysis \cite{kaiser_data-driven_2021, bollt_matching_2018, bollt_geometric_2021},
\begin{align}
\frac{dx}{dt} &= x^2, \quad x(0) = x_0.
\end{align}
The solution derived from separation of variables is
\begin{align}
    x(t) = \frac{x_0}{1-x_0 t}.
\end{align}
Alternatively we can solve for $x(t)$ using the Koopman approach by first mapping $x$ to a linear space.  We choose our eigenvalue to be $\lambda = -1$ and solve for the corresponding eigenfunction,
\begin{align}
    \varphi(x) &= e^{\int \frac{\lambda}{f(x)} dx} =  e^{\int \frac{-1}{x^2}dx} = e^{\frac{1}{x}} \label{eq:oneD_ex1_gx}\\
    \varphi(t) &= \varphi_0 e^{\lambda t} = e^{\frac{1}{x_0}}e^{-t} = e^{\frac{1-x_0 t}{x_0}}.
\end{align}
Solving for $x$ in terms of $\varphi$, using Eq.~\ref{eq:oneD_ex1_gx} gives us
\begin{align}
    x(t) = \frac{1}{\ln[\varphi(t)]} = \frac{1}{\ln[e^{\frac{1-x_0 t}{x_0}}]} = \frac{x_0}{1-x_0 t}.
\end{align}
Although the Koopman approach is less efficient than solving via separation of variables in the 1-dimensional case, we will use a Koopman approach to solve 2-dimensional systems which cannot be solved directly.

\subsection{1-dimensional ODE --- example 2(a)}

\begin{figure}[t]
    \centering
    \includegraphics[width=0.7\linewidth]{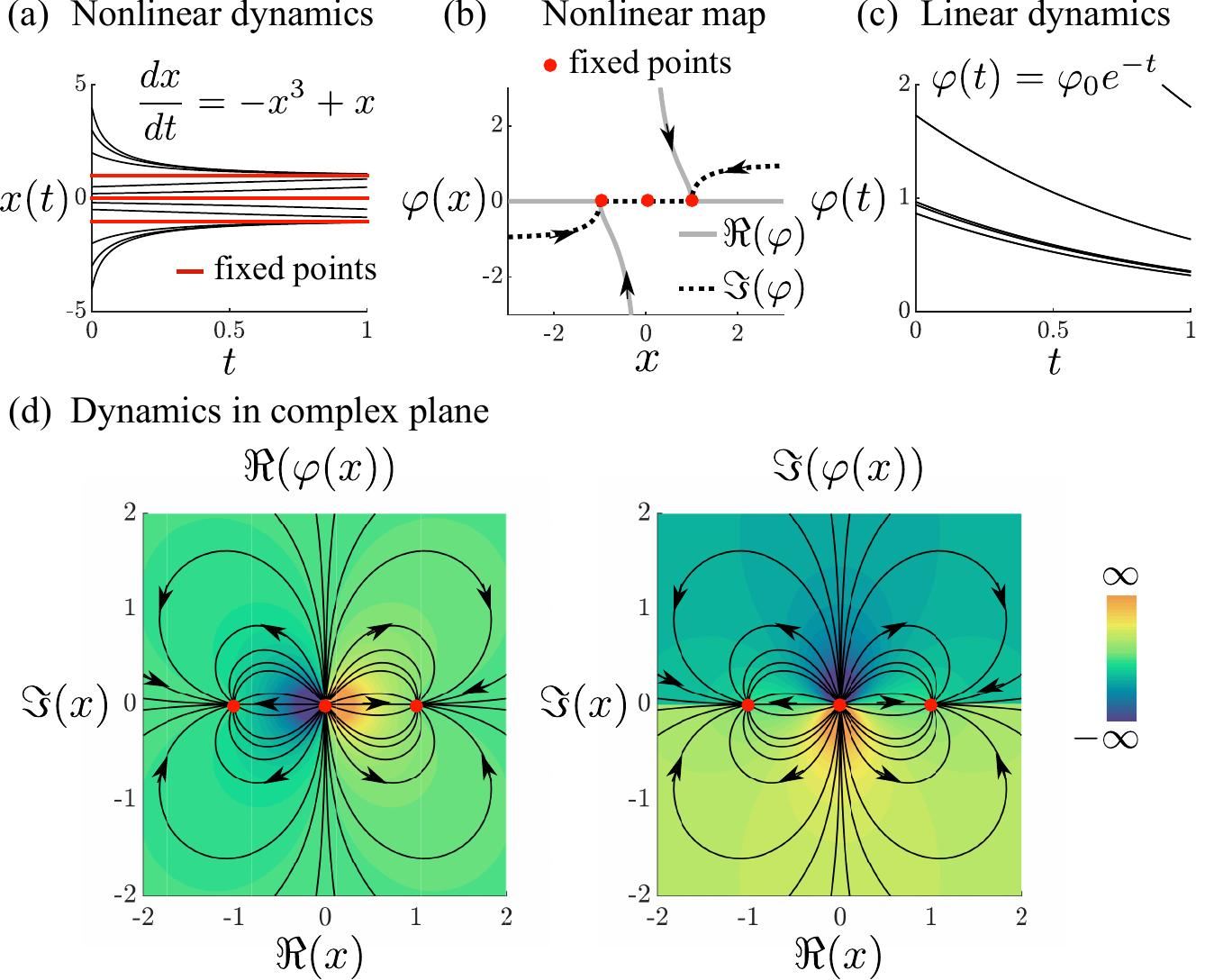}
    \caption{(a) Eq.~\ref{eq:oneD_ex2_dxdt} has nonlinear dynamics. (b) The nonlinear dynamics of $x$ can be mapped to eigenfunction $\varphi(x)$. (c) The dynamics of $\varphi$ is linear. (d) Dynamics of $x = a + ib$ viewed in the complex plane as well as mapping to complex-valued $\varphi$}
    \label{fig:oneD_ex2}
\end{figure}

Suppose we have a nonlinear differential equation of the form
\begin{align}
\frac{dx}{dt} &= -x^3+x, \quad x(0) = x_0. \label{eq:oneD_ex2_dxdt}
\end{align}
The solution can be found using separation of variables, resulting in
\begin{align}
    x(t) = \frac{\sign(x_0)e^t}{\sqrt{-1+e^{2t}+\frac{1}{x_0^2}}}. \label{eq:oneD_ex2_soln_sov}
\end{align}
Alternatively, we can also solve for $x(t)$ via the Koopman approach by first mapping $x$ to a variable that has linear dynamics, solving the linear ordinary differential equation, and then mapping the solution back to the nonlinear space.
First let us factor the right hand side of the differential equation,
\begin{align}
    \frac{dx}{dt} &= -x(x+1)(x-1).
\end{align}
The fixed points of the nonlinear ODE are used to construct the solution.
The system has two stable fixed points at $x = \pm 1$ and a source at $x=0$ (Fig.~\ref{fig:oneD_ex2}(a)). 
We can map both stable fixed points to the fixed point in the Koopman linear space by setting the eigenvalue to be $\lambda = -1$. The unstable fixed point is mapped to infinity (Fig.~\ref{fig:oneD_ex2}(b)).
We solve for the eigenfunction using the steps outlined above,

\begin{align}
    \varphi(x) &= \exp \int \frac{\lambda}{f(x)}dx \nonumber \\
    &= \exp \int \frac{-1}{-x(x^2-1)}dx \nonumber  \\
    &= \exp \left( \int \frac{-1}{x}dx + \int \frac{x}{x^2-1}dx \right) \nonumber  \\
    &= \exp [-\ln(x) + \frac{1}{2}\ln(1-x^2)] \nonumber \\
    \varphi(x) &=\frac{\sqrt{1-x^2}}{x}, \quad \lambda = -1. \label{eq:1dex2_gx}
\end{align}
The initial condition $x_0$ mapped to the eigenfunction space $\varphi(x)$ is 
\begin{align}
    \varphi_0 =\frac{\sqrt{1-x_0^2}}{x_0}.
\end{align}
The dynamics of $\varphi$ are linear (Fig.~\ref{fig:oneD_ex2}(c)). Therefore the solution for $\varphi(t)$ is
\begin{align}
    \varphi(t) = \varphi_0 e^{\lambda t} = \frac{\sqrt{1-x_0^2}}{x_0} e^{- t}.
\end{align}
Using Eq.~\ref{eq:1dex2_gx} we solve for $x$ in terms of $\varphi$,
\begin{align}
    x(t) = \frac{\sign(x_0)}{\sqrt{1+\varphi^2(t)}}= \frac{\sign(x_0)}{\sqrt{1+\frac{1-x_0^2}{x_0^2}e^{-2t}}}.
\end{align}
The Koopman-derived solution is equivalent to the solution derived using separation of variables, Eq.~\ref{eq:oneD_ex2_soln_sov}.
In a previous Koopman approach to solving Eq.~\ref{eq:oneD_ex2_dxdt}, separate linearization transforms were computed for each basin of attraction centered at each fixed point \cite{lan_linearization_2013}.
In contrast, we create a single nonlinear-to-linear mapping that is applicable to the entire domain of the nonlinear ODE.

The dynamics of nonlinear differential equations can be understood more fully by extending the dynamics into the complex plane. While all the fixed points of Eq.~\ref{eq:oneD_ex2_dxdt} are real, other differential equations have complex fixed points that impact the dynamics. Understanding the dynamics around fixed points is key to understanding the dynamics as a whole.
We allow $x$ to be a complex variable $x = a+ ib$ and solve for the dynamics of the real and complex component of $x$,
\begin{align}
\begin{split}
    \frac{da}{dt} &= a - a^3 + 3ab^2\\
    \frac{db}{dt} &= b - 3 a^2b + b^3.
\end{split}
\end{align}
The dynamics in the complex plane is an extension of the dynamics along the real line. Figure~\ref{fig:oneD_ex2}(d) shows the dynamics in the complex plane; the linear dynamics of $\varphi(x)$ hold for complex values of $x$. By slicing the dynamics where the imaginary component is zero, $b=0$, we recover the dynamics along the real line (Fig.~\ref{fig:oneD_ex2}(b)). We see that $\lim_{x \rightarrow \pm 1} \varphi(x) = 0$ and $\lim_{x \rightarrow 0} \varphi(x) = \pm \infty$.

\subsection{1-dimensional ODE --- example 2(b)}

\begin{figure}[t]
    \centering
    \includegraphics[width=0.70\linewidth]{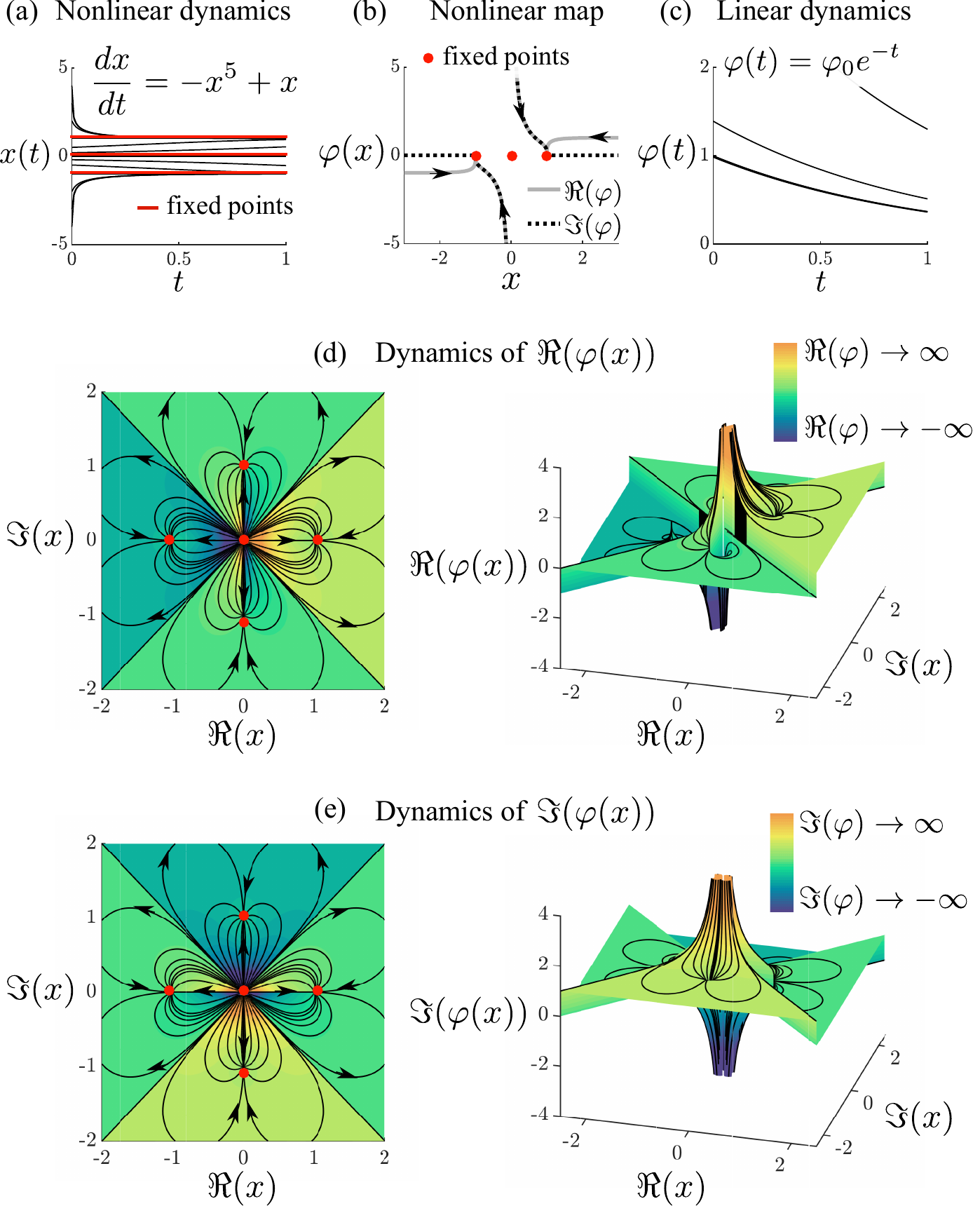}
    \caption{(a) Eq.~\ref{eq:oneD_ex22_dxdt} has nonlinear dynamics. (b) $x$ mapped to complex-valued eigenfunction $\varphi(x)$ (c) The dynamics of $\varphi$ is linear. (d-e) Dynamics of $x = a + ib$ viewed in the complex plaine as well as mapped to complex-valued $\varphi$}
    \label{fig:oneD_ex22}
\end{figure}

Now suppose we take Eq.~\ref{eq:oneD_ex2_dxdt} and add two more fixed points to the system but on the imaginary axis, still resulting in a real-valued ODE,
\begin{align}
    \frac{dx}{dt} &= -x^5 + x, \quad x(0) = x_0 \label{eq:oneD_ex22_dxdt}\\
    \frac{dx}{dt} &= (-x^3 + x)(x^2+1) \nonumber \\
    \frac{dx}{dt} &= -x(x+1)(x-1)(x+i)(x-i). \nonumber
\end{align}
We use the eigenvalue $\lambda = -1$ to map all stable fixed points, $x = \pm 1, \pm i$, to the stable fixed point in the eigenfunction space $\varphi =0$ (Fig.~\ref{fig:oneD_ex22}(a-b)). We map the unstable fixed point, $x = 0$, to unstable fixed point $\varphi = \infty$,
\begin{align}
    \varphi(x) = x^{-1}(x-1)^{\frac{1}{4}}(x+1)^{\frac{1}{4}}(x-i)^{\frac{1}{4}}(x+i)^{\frac{1}{4}} = \frac{(x^4-1)^{\frac{1}{4}}}{x}. \label{eq:1dex22_gx}
\end{align}
The initial condition $x_0$ mapped to $\varphi(x)$ is $\varphi_0 =\frac{(x_0^4-1)^\frac{1}{4}}{x_0}$.
The solution for $\varphi(t)$ is
\begin{align}
    \varphi(t) = \varphi_0 e^{\lambda t} = \frac{(x_0^4-1)^\frac{1}{4}}{x_0} e^{- t}.
\end{align}
Using Eq.~\ref{eq:1dex22_gx} we solve for $x$ in terms of $\varphi$,
\begin{align}
    x(t) = \frac{1^{\frac{1}{4}}}{(1-\varphi^4)^{\frac{1}{4}}} = \frac{1^{\frac{1}{4}}}{\left( 1-\left(\frac{(x_0^4-1)^\frac{1}{4}}{x_0} e^{- t} \right)^4 \right)^{\frac{1}{4}}}.
\end{align}
Figure~\ref{fig:oneD_ex22}(b-c) shows that the nonlinear mapping of the eigenfunction $\varphi(x)$ generates linear dynamics for $\varphi(t)$.
The dynamics when $x \in \mathds{C}$, where $x = a + ib$, can be written as
\begin{align}
\begin{split}
    \frac{da}{dt} &= a - a^5 + 10a^3b^2 - 5ab^4\\
    \frac{db}{dt} &= b - 5a^4b + 10a^2b^3-b^5.
\end{split}
\end{align}
We observe that Eq.~\ref{eq:oneD_ex22_dxdt} must map to a complex-valued eigenfunction in order to obtain linear dynamics. If $\dot{x} = f(x)$ has complex fixed points, then $\varphi(x)$ will be a complex-valued function even for real inputs $x \in \mathds{R}$. Figures~\ref{fig:oneD_ex22}(d-e) shows $x = \Re(x)+ i\Im(x)$ mapped to the real and imaginary components of the eigenfunction $\varphi(x) = \Re(\varphi) + i \Im(\varphi)$.

\subsection{1-dimensional ODE --- example 3}

Consider the following nonlinear differential equation and its factored form:
\begin{align}
\frac{dx}{dt} &= x^3+2x^2+2x, \quad x(0) = x_0 \label{eq:oneD_ex3_dxdt}\\ 
\frac{dx}{dt} &= x(x-(-1-i))(x-(-1+i)). \nonumber
\end{align}
Set $\lambda = 1$, then using Eq.~\ref{eq:final_g} we get
\begin{align}
    \varphi(x) &= x^{\frac{1}{2}} [x-(-1-i)]^{\frac{1}{-2+2i}} [x-(-1+i)]^{\frac{1}{-2-2i}} = \frac{e^{-\frac{1}{2} \arctan(1+x)} \sqrt{x}}{(2+2x+x^2)^\frac{1}{4}}. \label{eq:gx_1D_ex3} 
\end{align}
The dynamics of $\varphi$ are $\varphi(t) = \varphi_0 e^{t}$
with initial condition
\begin{align}
    \varphi_0(x_0) &= \frac{e^{-\frac{1}{2} \arctan(1+x_0)} \sqrt{x_0}}{(2+2x_0+x_0^2)^\frac{1}{4}}.
\end{align}
The mapping from $x$ to $\varphi$, Eq.~\ref{eq:gx_1D_ex3}, does not have a nice inverse; we cannot find an explicit solution for $x$, only an implicit solution.
The solution for $x(t)$ is the implicit solution to
\begin{align}
    \frac{e^{-\frac{1}{2} \arctan(1+x_0)} \sqrt{x_0}}{(2+2x_0+x_0^2)^\frac{1}{4}} e^t =\frac{e^{-\frac{1}{2} \arctan(1+x(t))} \sqrt{x(t)}}{(2+2x(t)+x(t)^2)^\frac{1}{4}}.
\end{align}
Implicit solutions can be solved using a numerical method.


\section{Koopman eigenfunctions for 2-dimensional ODEs}\label{sec:2D_examples}

We now turn our attention to solving 2-dimensional nonlinear ODEs using Algorithm~\ref{alg:koopman_method}. Not all 2-dimensional ODEs can be solved using this method.
We must have multiple real invariant manifolds and we must be able to combine the invariant manifold generating functions in such a manner that the resulting functions are eigenfunctions.
In order to solve a 2-dimensional autonomous ODE using the following method, the ODE must satisfy three requirements
\begin{enumerate}
    \item The ODE must have at least two real invariant manifolds $M_i(\b{x})$.
    \item There must exist at least two linear combinations of $N$-functions that result in constants.
    \item At least two of the weight vectors, $\b{p}_i$, corresponding to eigenfunctions $\varphi_i$, must be linearly independent; this requirement guarantees the eigenfunctions generated from the $M$-functions are from different equivalence classes.
\end{enumerate}
Once we have found a pair of independent eigenfunctions we can use Algorithm~\ref{alg:koopman_method} to solve the nonlinear ODE.

\begin{algorithm}[t]
\caption{Koopman method for solving nonlinear ODEs in $\mathds{R}^2$} \label{alg:koopman_method}
\begin{enumerate}
    \item Find independent eigenfunctions: $(x,y) \mapsto (\varphi_1(x,y),\varphi_2(x,y))$
    \item Solve eigenfunction dynamics: $\varphi(t) = \varphi_0 e^{\lambda t}$
    \item Compute initial condition: $\varphi_0 = \varphi(x_0,y_0)$
    \item Solve for original variables:  $(\varphi_1,\varphi_2) \mapsto (x(\varphi_1,\varphi_2),y(\varphi_1,\varphi_2))$
    \item Substitute eigenfunction solutions into solutions for original variables:\\
    $x(t) = x(\varphi_1(t),\varphi_2(t)) = x(\varphi_1(x_0,y_0)e^{\lambda_1 t},\varphi_2(x_0,y_0)e^{\lambda_2 t})$\\
    $y(t) = y(\varphi_1(t),\varphi_2(t)) = y(\varphi_1(x_0,y_0)e^{\lambda_1 t},\varphi_2(x_0,y_0)e^{\lambda_2 t})$
\end{enumerate}

\end{algorithm}

\subsection{Linear systems}

\begin{figure}[t]
    \centering
    \includegraphics[width=0.8\linewidth]{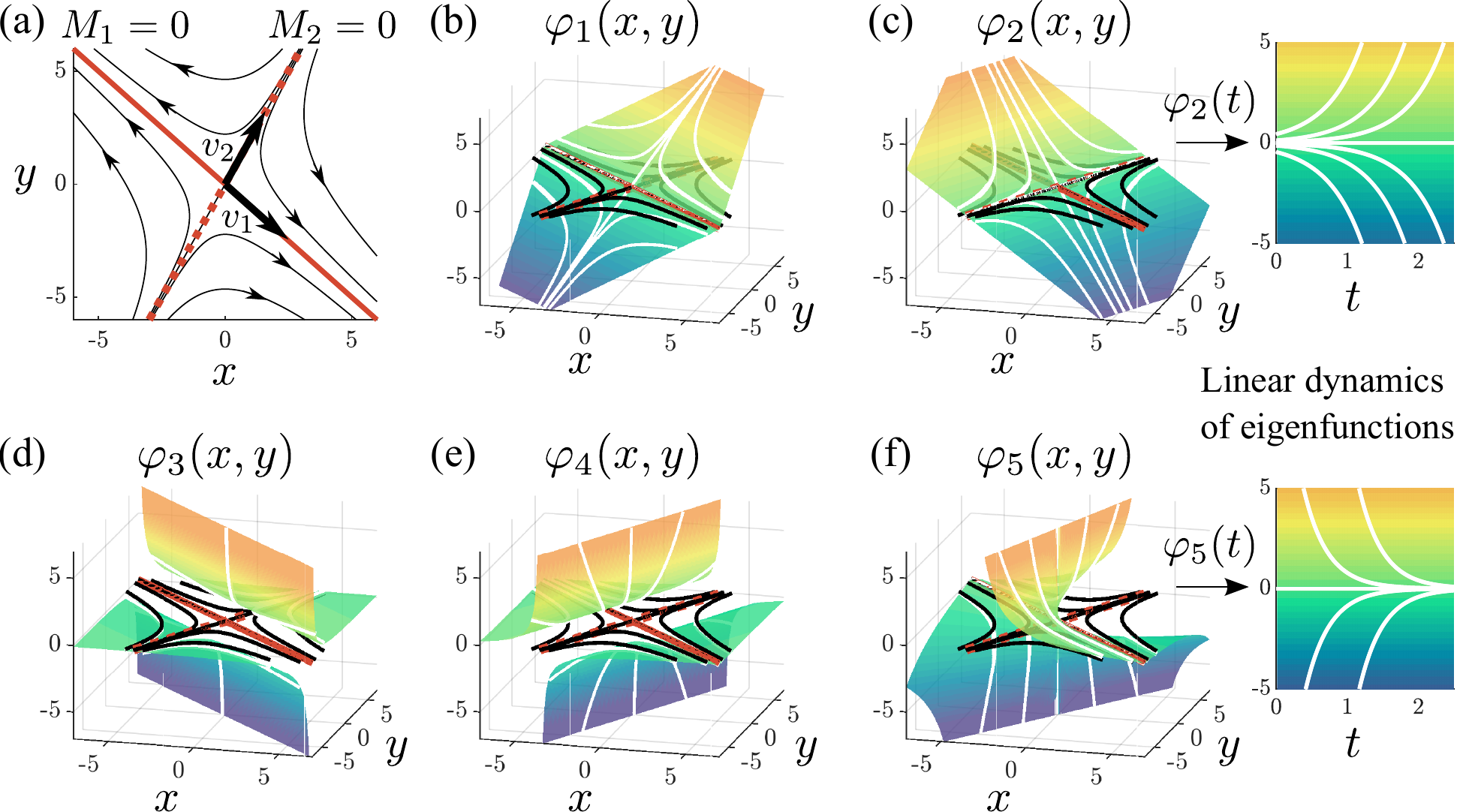}
    \caption{Linear dynamics and eigenfunctions (a) Phase plane of Eq.~\ref{eq:linear} (b) $\varphi_1(x,y)$, Eq.~\ref{eq:linear_g1}, formed from the first eigenvector (c) $\varphi_2(x,y)$, Eq.~\ref{eq:linear_g2}, formed from the second eigenvector (d-f) Additional eigenfunctions. The eigenfunctions have linear dynamics.}
    \label{fig:twoD_linear_real}
\end{figure}

Linear systems have well-known analytical solutions and are used ubiquitously in the applied sciences for prediction and control \cite{chicone_ordinary_2006,strogatz_nonlinear_2016, sontag_mathematical_2013}.
Although solutions to linear systems are well known, we will revisit the method here in a way that highlights the Koopman eigenfunctions, invariant manifolds, and their connection to the ensuing solution. Previous work has considered linear systems from a Koopman perspective \cite{budisic_applied_2012, bollt_geometric_2021, mezic_koopman_2017, mezic_spectrum_2020}; we consider linear systems here once again as an introduction for solving nonlinear systems.

The solution to a linear ordinary differential equation $\dot{\b{x}} = \b{A} \b{x}$ is $\b{x}(t) = \exp(\b{A}t) \b{x}_0$.
The solution is typically constructed by finding the eigenvalues and eigenvectors and then linearly composing them \cite{strogatz_nonlinear_2016, chicone_ordinary_2006}. We will instead use the Koopman approach, Algorithm~\ref{alg:koopman_method}, to find a solution.
Consider the 2D linear ordinary differential equation
\begin{align}
\begin{split} \label{eq:linear}
    \dot{x} &= x-y \\
    \dot{y} &= -2x
\end{split}
\end{align}
with initial conditions $x(0) = x_0$ and $y(0) = y_0$.
This system can be written as
\begin{align}
\dot{\b{x}} &= \b{A} \b{x}\\
    \begin{bmatrix}
    \dot{x}\\
    \dot{y}
    \end{bmatrix}
    &=
    \begin{bmatrix}
    1 & -1\\
    -2 & 0
    \end{bmatrix}
        \begin{bmatrix}
    x\\
    y
    \end{bmatrix}.
\end{align}
The eigenvalues and eigenvectors of $\b{A}$ are $\lambda_1 = 2$, $\lambda_2 = -1$, $v_1 = [-1 \ \ 1]^T$, and $v_2 = [1 \ \ 2]^T$. We use the eigenvectors to construct two invariant manifold generating functions, $M_1 = y+x$ and $M_2 = y-2x$ that are zero along the eigenvector directions; $\Lambda_1 = \{(x,y): M_1(x,y)=0 \}$ and $\Lambda_2 = \{(x,y): M_2(x,y)=0 \}$ are the two invariant manifolds of Eq.~\ref{eq:linear} that pass through the fixed point at the origin (Fig.~\ref{fig:twoD_linear_real}(a)). We can confirm that these are invariant manifolds of Eq.~\ref{eq:linear} by checking that they satisfy Eq.~\ref{eq:invar_man_req}, which also allows us to solve for the $N$-functions.
\begin{align*}
    \frac{d}{dt}M_1(x,y) &= \frac{d(y+x)}{dt}=-2x+x-y= (x+y)(-1)=M_1(x,y)N_1(x,y)\\
     \frac{d}{dt}M_2(x,y) &=\frac{d (y-2x)}{dt}=-2x-2(x-y)=(y-2x)2=M_2(x,y)N_2(x,y)
\end{align*}
Because $N_1$ and $N_2$ are both constants, $M_1(x,y)$ and $M_2(x,y)$ are both eigenfunctions of the linear system Eq.~\ref{eq:linear} with eigenvalues $ \lambda_1 = N_1(x,y) = -1$ and $\lambda_2 = N_2(x,y) = 2$,
\begin{alignat}{2}
    \varphi_1(x,y) &= y+x, \quad &&\lambda_1 = -1 \label{eq:linear_g1}\\
    \varphi_2(x,y) &= y-2x, \quad  &&\lambda_2 = 2. \label{eq:linear_g2}
\end{alignat}
Figure~\ref{fig:twoD_linear_real}(b-c) shows eigenfunctions $\varphi_1$ and $\varphi_2$ with the dynamics in the $(x,y)$ space projected onto the eigenfunction surfaces. The dynamics in the eigenfunction space are linear.
The weight vector of the $N$-functions for $\varphi_1$ is $\b{p}_1 = [1 \ \ 0]^T$ since $\lambda_1 = 1 N_1 + 0 N_2$, while the weight vector for $\varphi_2$ is $\b{p}_2 = [0 \ \ 1]^T$ since $\lambda_2 = 0 N_1 + 1 N_2$.
$\b{p}_1$ and $\b{p}_2$ are linearly independent; therefore, according to Theorem~\ref{thm:lin_independence}, $\varphi_1$ and $\varphi_2$ belong to different equivalence classes and can be used in conjunction to solve for $\b{x}(t)$.
We can confirm that these are in fact eigenvalue-eigenfunction pairs of Eq.~\ref{eq:linear} by checking that they satisfy Eq.~\ref{eq:lin_dyn_varphi_pde}.
\begin{align*}
    \nabla_{\b{x}} \varphi_1(\b{x})  \cdot  F(\b{x}) &=\lambda_1 \varphi_1(\b{x})\\
    \begin{bmatrix}
        1 & 1
    \end{bmatrix} \begin{bmatrix}
        x-y\\
        -2x
    \end{bmatrix}
   &=
    -1(y+x)\\
    -y-x &= -y-x\\
    &\implies  \varphi_1 \  \text{is an eigenfunction}.
\end{align*}
\begin{align*}
    \nabla_{\b{x}} \varphi_2(\b{x})  \cdot  F(\b{x}) &= \lambda_2 \varphi_2(\b{x})\\
    \begin{bmatrix}
        -2 & 1
    \end{bmatrix} \begin{bmatrix}
        x-y\\
        -2x
    \end{bmatrix}
   &=
    2(y-2x)\\
    2y-4x &= 2y-4x\\
    &\implies  \varphi_2 \  \text{is an eigenfunction}.
\end{align*}
According to Eq.~\ref{eq:varphi_t}, the eigenfunction solutions are
\begin{align}
    \varphi_1(t) &= \varphi_1(0) e^{- t} = \varphi_1(x_0,y_0) e^{- t} \label{eq:lin_varphi1}\\
     \varphi_2(t) &= \varphi_2(0) e^{2t} = \varphi_2(x_0,y_0) e^{2t}. \label{eq:lin_varphi2}
\end{align}
Eqs.~\ref{eq:linear_g1} and~\ref{eq:linear_g2} give us initial conditions in terms of $x$ and $y$, $\varphi_1(x_0,y_0) = y_0-x_0$, and $\varphi_2(x_0,y_0) = y_0-2x_0$.
Now that we have analytical solutions to the dynamics in the eigenfunction space we can map the solutions back to the original $(x,y)$ space.
We use the system of equations~\ref{eq:linear_g1} and~\ref{eq:linear_g2} to solve for $x$ and $y$ as functions of $\varphi_1$ and $\varphi_2$,
\begin{align}
    x(t) &= \varphi_1(t) - \varphi_2(t) \label{eq:linear_x(g)}\\
    y(t) &= 2\varphi_1(t) - \varphi_2(t). \label{eq:linear_y(g)}
\end{align}
We have analytical solutions for $\varphi_1$ (Eq.~\ref{eq:lin_varphi1}) and $\varphi_2$ (Eq.~\ref{eq:lin_varphi2}) and may substitute these solutions into the equations for $x$ and $y$.
With this substitution we get the analytical solution for ($x(t),y(t)$) in terms of time $t$ and the initial condition $(x_0, y_0)$,
\begin{align}
        x(t) &= (y_0-x_0)e^{2 t} - (y_0-2x_0)e^{- t}\\
    y(t) &= 2(y_0-x_0)e^{2 t} - (y_0-2x_0)e^{- t}.
\end{align}
The key to obtaining a solution was finding eigenfunctions, which have linear dynamics. Because observables with linear dynamics have known solutions, this process allows us to find an analytical solution in the ($\varphi_1(t), \varphi_2(t)$) space and then map the solution back to the ($x(t),y(t)$) space.

Although we are finished solving this problem, we will consider what other eigenfunctions we might have used, as $\varphi_1$ and $\varphi_2$ were not our only options.
Bollt (2021) demonstrates that additional eigenfunctions can be generated from primary eigenfunctions \cite{bollt_geometric_2021}; the inverse and product of eigenfunctions are also eigenfunctions.
From these properties we generate additional eigenfunctions for Eq.~\ref{eq:linear} out of eigenfunctions $\varphi_1$ and $\varphi_2$,
\begin{align}
    \varphi_3 &= \frac{1}{\varphi_1}, \quad \lambda_3 = -\lambda_1\\
    \varphi_4 &= \frac{1}{\varphi_2}, \quad \lambda_4 = -\lambda_2\\
    \varphi_5 &= \frac{\varphi_1}{\varphi_2}, \quad \lambda_5 = \lambda_1 + \frac{1}{\lambda_2}.
\end{align}
Figure~\ref{fig:twoD_linear_real} shows the phase plane of Eq.~\ref{eq:linear}, the dynamics mapped onto the eigenfunctions, and the resulting linear dynamics that occur on these special observables. $\varphi_1$ and $\varphi_2$ are both a linear mapping from the $(x,y)$ plane with a null space along $v_1$ and $v_2$ respectively. Conversely, $\varphi_3$ and $\varphi_4$ are nonlinear with respect to $x$ and $y$. Both of these functions have a discontinuity along the eigenvector directions; the eigenfunctions go to infinity as they approach the invariant manifolds. $\varphi_5$ has a null space along $v_1$ and is undefined along $v_2$ as it is composed of $\varphi_1$ in the numerator and $\varphi_2$ in the denominator. Although $\varphi_3$, $\varphi_4$, and $\varphi_5$ are all \emph{nonlinear} with respect to $x$ and $y$, they are \emph{linear} with respect to time. This means that they all display exponential growth or exponential decay as the eigenfunction observables are measured along the flow of the original variables.
We extend this process to nonlinear ODEs and use the invariant manifolds of nonlinear systems to generate eigenfunctions.


\subsection{Nonlinear example 1 --- linear invariant manifolds}

\begin{figure}[t]
    \centering
    \includegraphics[width=0.9\linewidth]{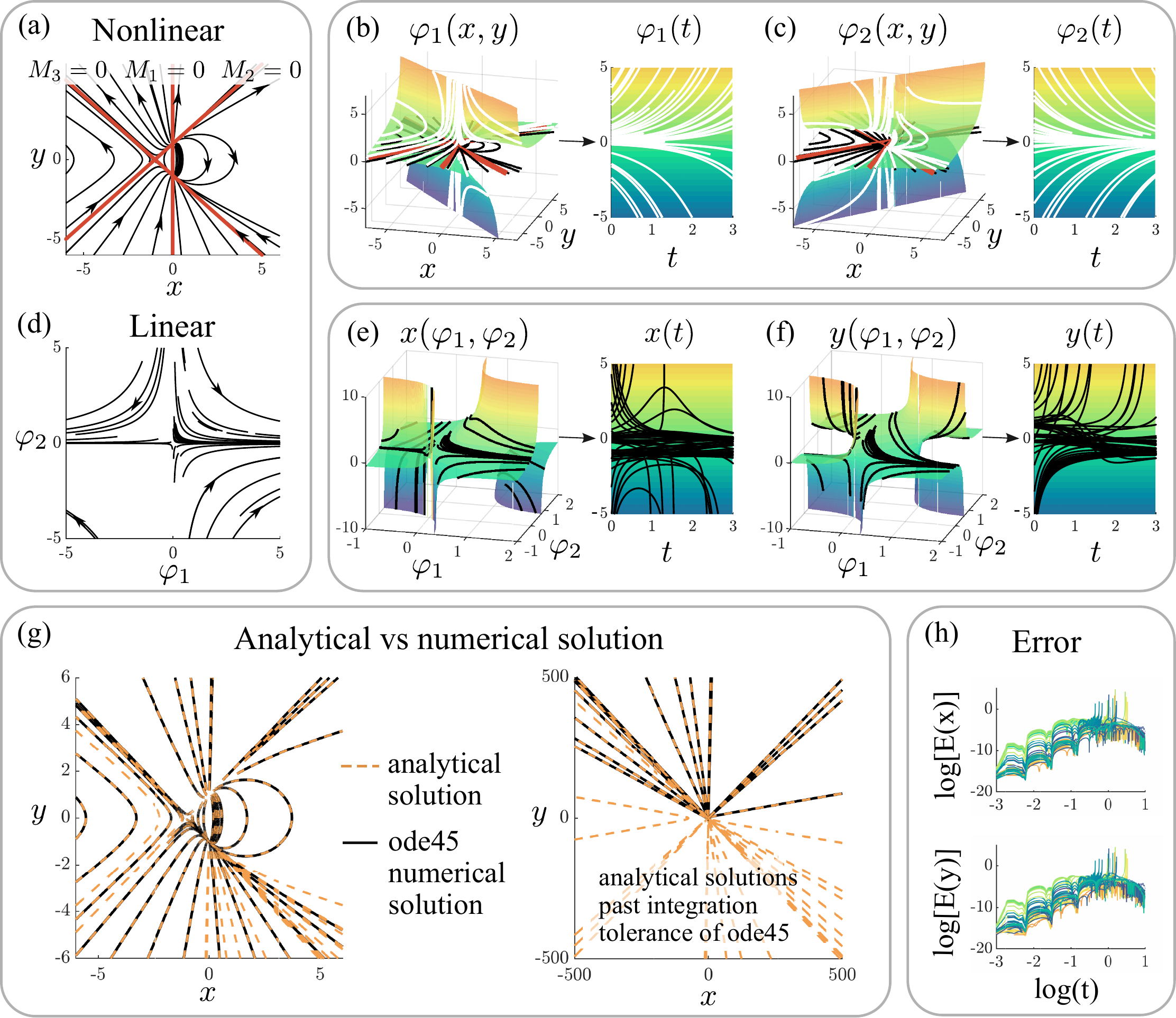}
    \caption{Dynamics in original space and eigenfunction space. (a) Phase plane for Eq.~\ref{eq:nonlinear_ex1} (b-c) Flow lines projected onto $\varphi_1$ and $\varphi_2$. (d) Phase plane of ($\varphi_1,\varphi_2$) (e-f) Flow lines of ($\varphi_1, \varphi_2$) mapped back to original variables ($x,y$) (g) Analytical solution compared to numerical solution produced using ode45. (h) Error in trajectories over time}
    \label{fig:nonlinear_ex1}
\end{figure}

We extend the method we used to solve the linear ODE to nonlinear ODEs.
Consider the nonlinear system
\begin{align}
\begin{split}\label{eq:nonlinear_ex1}
    \dot{x} &= xy\\
    \dot{y} &= y^2 -x -1
\end{split}
\end{align}
with initial conditions $x(0) = x_0$ and $y(0) = y_0$.
The fixed points of this system are at $(x,y) = (0,\pm 1), (-1,0)$.
The surfaces $M_1(x,y) = x$, $M_2(x,y) = y-x-1$, and $M_3(x,y) = y+x+1$ generate invariant manifolds $\Lambda_1 = \{(x,y): M_1(x,y)=0  \}$, $\Lambda_2 = \{(x,y): M_2(x,y)=0  \}$, and $\Lambda_3 = \{(x,y): M_3(x,y)=0  \}$ that go through the fixed points
 (Fig.~\ref{fig:nonlinear_ex1}(a)). 
We can confirm that these are indeed invariant manifolds of Eq.~\ref{eq:nonlinear_ex1} by checking that they satisfy Eq.~\ref{eq:invar_man_req} which also allows us to solve for the $N$-functions.
\begin{align*}
    \frac{d}{dt}M_1(x,y) &= \frac{dx}{dt}=xy=M_1 N_1\\
     \frac{d}{dt}M_2(x,y) &=\frac{d (y-x-1)}{dt}=y^2-x-1 - xy=(y-x-1)(y+1)=M_2 N_2\\
     \frac{d}{dt}M_3(x,y)&=\frac{d (y+x+1)}{dt} = y^2-x-1+xy=(y+x+1)(y-1) = M_3 N_3
\end{align*}
The $N$-functions corresponding to the $M$-functions are $N_1(x,y) = y$, $N_2(x,y) = y+1$, and $N_3(x,y)=y-1$. None of the $N$-functions are constants, therefore none of the $M$-functions are eigenfunctions. However, multiple linear combinations of the $N$-functions result in constants,
\begin{alignat*}{2}
    \lambda_1 = N_1(x,y) - N_3(x,y) &= y-(y-1) &&= 1\\
    \lambda_2 = N_1(x,y) - N_2(x,y) &= y-(y+1) &&= -1\\
    \lambda_3 = N_2(x,y) - N_3(x,y) &= y+1-(y-1) &&= 2.
\end{alignat*}
Therefore, according to Theorem~\ref{thm:manifolds_general}, we can construct eigenfunctions from the $M$-function quotients
\begin{align}
    \varphi_1(x,y) &= \frac{M_1}{M_3} = \frac{x}{1+x+y}, \quad \lambda_1 = 1 \label{eq:nonlin_ex1_g1}\\
    \varphi_2(x,y) & = \frac{M_1}{M_2} = \frac{x}{1+x-y},  \quad \lambda_2 = -1 \label{eq:nonlin_ex1_g2}\\
     \varphi_3(x,y) & = \frac{M_2}{M_3} = \frac{y-x-1}{y+x+1},  \quad \lambda_3 = 2.
\end{align}
The $N$-function weighting vectors for $\varphi_1$, $\varphi_2$, and $\varphi_3$ are $\b{p}_1 = [1 \ \ 0 \  -1]^T$, $\b{p}_2 = [1 \ -1 \ \ 0]^T$, and $\b{p}_3 = [0 \ \ 1 \ -1]^T$ since $\lambda_1 = 1N_1 + 0N_2 -1 N_3$, $\lambda_2 = 1N_1 - 1N_2 + 0N_3$, and $\lambda_3 = 0N_1 + 1N_2 - 1N_3$. 
All three of the $\b{p}$ vectors are linearly independent. Therefore, according to Theorem~\ref{thm:lin_independence} $\varphi_1$, $\varphi_2$, and $\varphi_3$ are in different equivalence classes and so any pair of these eigenfunctions can be used to solve for $\b{x}(t)$.
We confirm that these are eigenvalue-eigenfunction pairs of Eq.~\ref{eq:nonlinear_ex1} by checking that they satisfly Eq.~\ref{eq:lin_dyn_varphi_pde}.
\begin{align*}
    \nabla_{\b{x}} \varphi_1(\b{x})  \cdot  F(\b{x}) &=\lambda_1 \varphi_1(\b{x})\\
    \begin{bmatrix}
        \frac{1+y}{(1+x+y)^2} & \frac{-x}{(1+x+y)^2}
    \end{bmatrix} \begin{bmatrix}
        xy\\
        y^2-x-1
    \end{bmatrix}
   &=
    \frac{x}{1+x+y}\\
   \frac{x}{1+x+y} &= \frac{x}{1+x+y}\\
    &\implies  \varphi_1 \  \text{is an eigenfunction}.
\end{align*}

\begin{align*}
    \nabla_{\b{x}} \varphi_2(\b{x})  \cdot  F(\b{x}) &=\lambda_2 \varphi_2(\b{x})\\
    \begin{bmatrix}
        \frac{1-y}{(1+x-y)^2} & \frac{x}{(1+x-y)^2}
    \end{bmatrix} \begin{bmatrix}
        xy\\
        y^2-x-1
    \end{bmatrix}
   &=
    \frac{(-1)x}{1+x-y}\\
   \frac{-x}{1+x-y} &= \frac{-x}{1+x-y}\\
    &\implies  \varphi_2 \  \text{is an eigenfunction}.
\end{align*}
The dynamics on these mappings are linear, giving us solutions in the eigenfunction space,
\begin{align}
    \varphi_1(t) &= \varphi_1(0) e^{\lambda_1 t} = \varphi_1(x_0,y_0) e^{t}\\
    \varphi_2(t) &= \varphi_2(0) e^{\lambda_2 t} = \varphi_2(x_0,y_0) e^{-t}.
\end{align}
The initial conditions in the eigenfunction space can be found by mapping the initial conditions in the $(x,y)$ space, $(x_0,y_0)$, to the eigenfunction space,
\begin{align}
    \varphi_1(x_0,y_0) &= \frac{x_0}{1 + x_0 + y_0}\\
    \varphi_2(x_0,y_0) & = \frac{x_0}{1+x_0-y_0}.
\end{align}
We now need to map the solution back to the original $(x,y)$ space. We use the system Eqs.~\ref{eq:nonlin_ex1_g1} and \ref{eq:nonlin_ex1_g2} to solve for $(x, y)$. Notice that if these eigenfunctions are not independent we will not be able to map back to the $(x,y)$ space. This give us
\begin{align}
    x(\varphi_1,\varphi_2) &= \frac{2 \varphi_1 \varphi_2}{\varphi_1 + \varphi_2 - 2 \varphi_1 \varphi_2}\\
    y(\varphi_1,\varphi_2) &= \frac{-\varphi_1 + \varphi_2}{\varphi_1 + \varphi_2 - 2 \varphi_1 \varphi_2}.
\end{align}
Substituting the analytical solutions for $\varphi_1(t;x_0,y_0)$ and $\varphi_2(t;x_0,y_0)$ into the expressions for $x$ and $y$ gives us analytical solutions for $x$ and $y$,

\begin{align}
    x(t) &= \frac{2 \left( \frac{x_0}{1 + x_0 + y_0}\right) \left(\frac{x_0}{1+x_0-y_0} \right)}{\left( \frac{x_0}{1 + x_0 + y_0} \right) e^{t} + \left(\frac{x_0}{1+x_0-y_0} \right) e^{-t} - 2 \left( \frac{x_0}{1 + x_0 + y_0}\right) \left(\frac{x_0}{1+x_0-y_0} \right)}\\
    y(t) &= \frac{-\left(\frac{x_0}{1 + x_0 + y_0} \right) e^t + \left(\frac{x_0}{1+x_0-y_0} \right) e^{-t}}{\left(\frac{x_0}{1 + x_0 + y_0} \right) e^t + \left(\frac{x_0}{1+x_0-y_0} \right) e^{-t} - 2 \left(\frac{x_0}{1 + x_0 + y_0} \right) \left(\frac{x_0}{1+x_0-y_0} \right)}.
\end{align}
Notice that, similar to the linear case, the solution is constructed from combinations of exponentials originating from the linear eigenfunction solutions. Instead of linear combinations of exponentials, however, the solution is formed from a nonlinear combination of these exponentials.
Figure~\ref{fig:nonlinear_ex1}(b-c) shows the flow lines in the $(x,y)$ space projected onto the eigenfunctions, resulting in linear dynamics. The system $(\varphi_1,\varphi_2)$ has simple linear dynamics (Fig.~\ref{fig:nonlinear_ex1}(d)) which can be projected back to the original space (Fig.~\ref{fig:nonlinear_ex1}(e-f)). The analytical solution for $\b{x}(t)$ initially matches the numerical solution from ode45 (Fig.~\ref{fig:nonlinear_ex1}(g)), however, the error in the numerical solution accumulates over time (Fig.~\ref{fig:nonlinear_ex1}(h)). In addition to the accumulation of errors, ode45 cannot produce a numerical solution past the threshold of its integration tolerance. The analytical solution shows that the trajectories go to infinity and back in finite time, which is a phenomena that is unobservable using ode45. This is due to the cessation in integration when the integration tolerance is met (Fig.~\ref{fig:nonlinear_ex1}(g)).

\subsection{Nonlinear example 2}

\begin{figure}[t]
    \centering
    \includegraphics[width=0.9\linewidth]{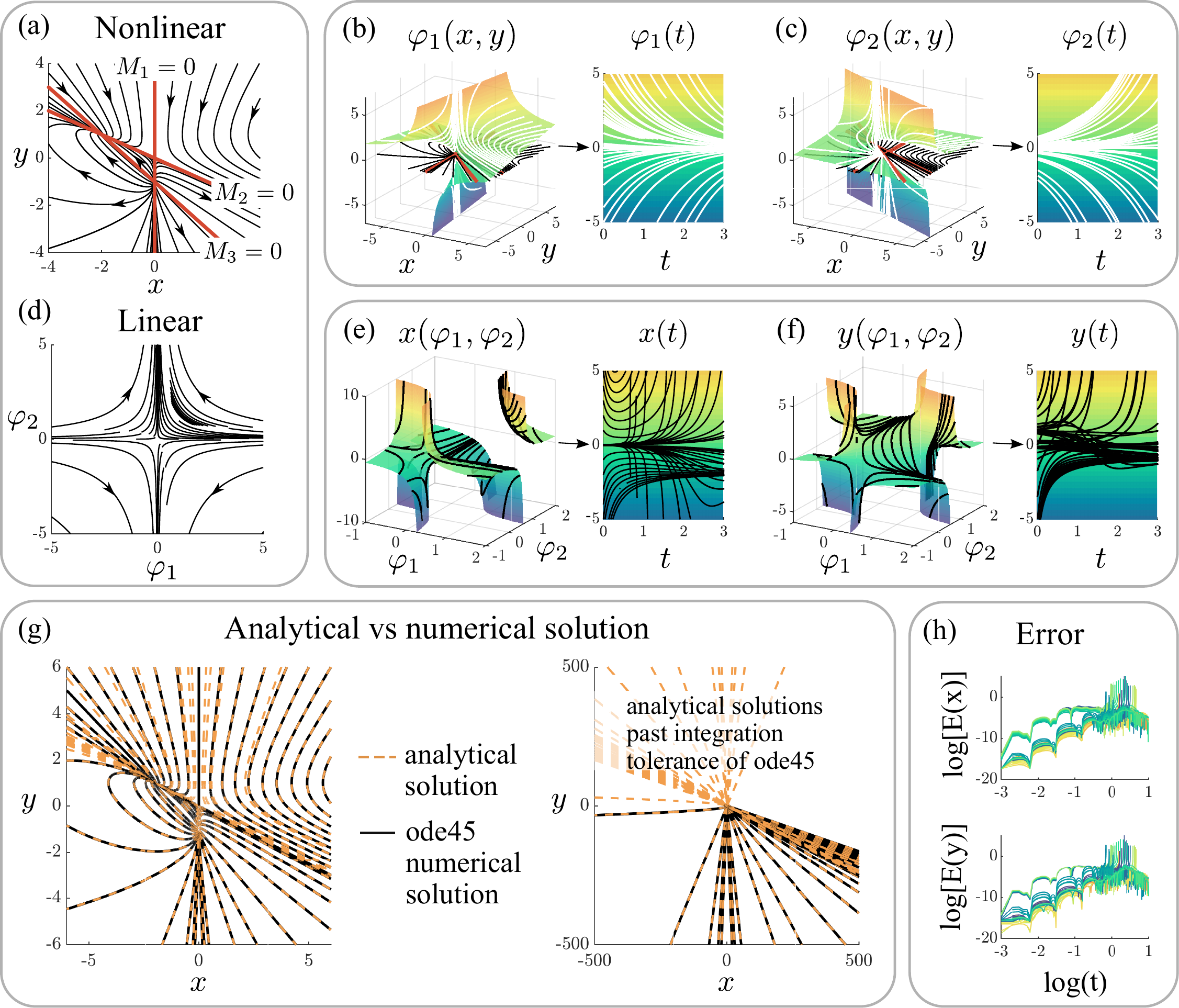}
    \caption{Dynamics in original space and eigenfunction space. (a) Phase plane for Eq.~\ref{eq:nonlinear_ex2} (b-c) Flow lines projected onto $\varphi_1$ and $\varphi_2$. (d) Phase plane of ($\varphi_1,\varphi_2$) (e-f) Flow lines of ($\varphi_1, \varphi_2$) mapped back to original variables ($x,y$) (g) Analytical solution compared to numerical solution produced using ode45. (h) Error in trajectories over time}
    \label{fig:nonlinear_ex2}
\end{figure}

Consider another nonlinear ODE that, as in the previous example, has three real linear invariant manifolds,
\begin{align} \label{eq:nonlinear_ex2}
\begin{split}
      \dot{x} &= x-xy\\
    \dot{y} &= -x-y-y^2  
\end{split}
\end{align}
with initial conditions $x(0) = x_0$ and $y(0) = y_0$.
The invariant manifold generating functions are $M_1 = x$, $M_2 = x+2y$, and $M_3 = 1+x+y$. The invariant manifolds that include the fixed points of the system are $\Lambda_1 = \{(x,y): M_1(x,y) = 0 \}$, $\Lambda_2 = \{(x,y): M_2(x,y) = 0 \}$, and $\Lambda_3 = \{(x,y): M_3(x,y) = 0 \}$ (Fig.~\ref{fig:nonlinear_ex2}(a)). 
We can confirm that the proposed $M$-functions are invariant manifolds of Eq.~\ref{eq:nonlinear_ex2} and solve for the $N$-functions by checking that they satisfy Eq.~\ref{eq:invar_man_req}.
\begin{align*}
    \frac{d}{dt}M_1(x,y) &= \frac{dx}{dt}=x-xy=x(1-y)=M_1 N_1\\
     \frac{d}{dt}M_2(x,y) &= \frac{d(x+2y)}{dt}=-x-xy-2y-y^2=(x+2y)(-1-y)=M_2 N_2\\
     \frac{d}{dt}M_3(x,y)&=\frac{d(1+x+y)}{dt}=-xy-y-y^2=(1+x+y)(-y)= M_3 N_3
\end{align*}
The $N$-functions corresponding to the $M$-functions are $N_1(x,y) = 1-y$, $N_2(x,y) = -1-y$, and $N_3(x,y) = -y$. Multiple differences of the $N$-functions result in constants,
\begin{alignat*}{2}
    \lambda_1 = N_3(x,y) - N_1(x,y) &= -y-(1-y) &&=-1\\
    \lambda_2 = N_3(x,y) - N_2(x,y) &= -y-(-1-y) &&=1\\
    \lambda_3 = N_1(x,y) - N_3(x,y) &= 1-y-(-y) &&= 1\\
    \lambda_4 = N_1(x,y) - N_2(x,y) &= 1-y-(-1-y) &&=2.
\end{alignat*}
Therefore, according to Theorem~\ref{thm:manifolds_general}, Eq.~\ref{eq:nonlinear_ex2} has eigenfunctions
\begin{alignat}{2}
    \varphi_1(x,y) &=\frac{M_3}{M_1} = \frac{1+x+y}{x}, \quad &&\lambda_1 = -1\\
    \varphi_2(x,y) &= \frac{M_3}{M_2}= \frac{1+x+y}{x+2y}, \quad &&\lambda_2 = 1\\
    \varphi_3(x,y) &= \frac{M_1}{M_3}= \frac{x}{1+x+y}, \quad &&\lambda_3 = 1\\
    \varphi_4(x,y) &= \frac{2 M_1}{M_2}= \frac{2x}{x+2y}, \quad &&\lambda_4 = 2.
\end{alignat}
The weighting vectors for these eigenfunctions are $\b{p}_1 = [-1 \ 0 \ 1]^T$, $\b{p}_2 = [0 \ -1 \ \ 1]^T$, $\b{p}_3 = [1 \ \ 0 \ -1]^T$, and $\b{p}_4 = [1 \ -1 \ \ 0]^T$ since $\lambda_1 = -1N_1 + 0 N_2 + 1 N_3$, $\lambda_2 = 0N_1 - 1N_2 + 1 N_3$, $\lambda_3 = 1N_1 + 0N_2 - 1 N_3$, and $\lambda_4 = 1N_1 - 1N_2 +0 N_3$.
Not all pairs of vectors are linearly independent. For example, $\b{p}_1$ and $\b{p}_3$ are linearly dependent, meaning that the eigenfunction pair ($\varphi_1, \varphi_3$) cannot be used to solve for $\b{x}(t)$. $\b{p}_1$ and $\b{p}_2$ are linearly independent; we therefore select the corresponding pair of eigenfunctions ($\varphi_1, \varphi_2$) to solve for $\b{x}(t)$.
Using $\varphi_1(x,y)$ and $\varphi_2(x,y)$ to solve for $x$ and $y$ gives us
\begin{align}
    x(\varphi_1,\varphi_2) = \frac{2\varphi_2}{-\varphi_1-\varphi_2+2\varphi_1\varphi_2}\\
    y(\varphi_1,\varphi_2) =\frac{\varphi_1 - \varphi_2}{-\varphi_1 - \varphi_2 + 2\varphi_1 \varphi_2}.
\end{align}
Substituting the solutions $\varphi_1(t) = \varphi_1(x_0,y_0)e^{\lambda_1 t}$ and $\varphi_2(t) = \varphi_2(x_0,y_0)e^{\lambda_2 t}$ for $\varphi_1$ and $\varphi_2$ in the formulas for $x$ and $y$ gives us
\begin{align}
    x(t) &= \frac{2\left(\frac{1+x_0+y_0}{x_0+2y_0} \right) e^{t}}{-\left(\frac{1+x_0+y_0}{x_0} \right)e^{-t}-\left(\frac{1+x_0+y_0}{x_0+2y_0} \right) e^{t}+2\left(\frac{1+x_0+y_0}{x_0} \right) \left(\frac{1+x_0+y_0}{x_0+2y_0} \right)}\\
     y(t) &= \frac{\left(\frac{1+x_0+y_0}{x_0} \right)e^{-t} - \left(\frac{1+x_0+y_0}{x_0+2y_0} \right) e^{t} }{-\left(\frac{1+x_0+y_0}{x_0} \right)e^{-t}-\left(\frac{1+x_0+y_0}{x_0+2y_0} \right) e^{t}+2\left(\frac{1+x_0+y_0}{x_0} \right) \left(\frac{1+x_0+y_0}{x_0+2y_0} \right)}.
\end{align}
Simplifying we get
\begin{align}
    x(t) &= \frac{2\left(\frac{1+x_0+y_0}{x_0+2y_0} \right) e^{t}}{-\left(\frac{1+x_0+y_0}{x_0} \right)e^{-t}-\left(\frac{1+x_0+y_0}{x_0+2y_0} \right) e^{t}+\frac{2(1+x_0+y_0)^2}{x_0(x_0+2y_0)} }\\
     y(t) &= \frac{\left(\frac{1+x_0+y_0}{x_0} \right)e^{-t} - \left(\frac{1+x_0+y_0}{x_0+2y_0} \right) e^{t} }{-\left(\frac{1+x_0+y_0}{x_0} \right)e^{-t}-\left(\frac{1+x_0+y_0}{x_0+2y_0} \right) e^{t}+\frac{2(1+x_0+y_0)^2}{x_0(x_0+2y_0)}}.
\end{align}
Figure~\ref{fig:nonlinear_ex2}(b-c) shows the flow of Eq.~\ref{eq:nonlinear_ex2} in the $(x,y)$ space projected onto the eigenfunctions which have linear dynamics. The flow in the $(\varphi_1, \varphi_2)$ space can be solved and projected back onto the original variables (Fig.~\ref{fig:nonlinear_ex2}(d-f)). The numerical solutions to $\b{x}(t)$ match the analytical solutions initially, however, error does accumulate; when the integration tolerance of ode45 is met the numerical solution ends (Fig.~\ref{fig:nonlinear_ex2}(g-h)).

\subsection{Nonlinear example 3 --- quasi-2D system}

This next example connects the 1-dimensional and 2-dimensional cases and clarifies the connection between the system's invariant manifolds and the eigenfunctions,
\begin{align}
\begin{split}
       \dot{x} &= x^2 - x\\
    \dot{y} &= xy - 2y 
\end{split}
\end{align}
with initial conditions $x(0) = x_0$ and $y(0) = y_0$. The surfaces that generate invariant manifolds that include the fixed points along their zero level-set are $M_1 = x$, $M_2 = x-1$, and $M_3 = y$; these invariant manifolds are $\Lambda_1 = \{(x,y): M_1(x,y)=0 \}$, $\Lambda_2 = \{(x,y): M_2(x,y)=0 \}$, and $\Lambda_3 = \{(x,y): M_3(x,y)=0 \}$. We show that these equations satisfy Eq.~\ref{eq:invar_man_req} and solve for the $N$-functions.
\begin{align*}
    \frac{d}{dt}M_1(x,y) &= \frac{dx}{dt} = x^2-x = x(x-1) = M_1 N_1 \\
    \frac{d}{dt}M_2(x,y) &= \frac{d(x-1)}{dt}= x^2-x = (x-1)x = M_2 N_2\\
    \frac{d}{dt}M_3(x,y) &= \frac{dy}{dt} = xy-2y = y(x-2) = M_3 N_3
\end{align*}
The $N$-functions corresponding to the $M$-functions are $N_1(x,y) = x-1$, $N_2(x,y) = x$, and $N_3(x,y) = x-2$. Multiple differences of the $N$-functions result in constants,
\begin{alignat*}{2}
    \lambda_1 = N_2(x,y) - N_1(x,y) &= x-(x-1) &&=1\\
   \lambda_2 = N_1(x,y) - N_3(x,y) &= x-1-(x-2) &&=1\\
   \lambda_3 = N_2(x,y) - N_3(x,y) &= x-(x-2) &&= 2.
\end{alignat*}
According to Theorem~\ref{thm:manifolds_general}, we can construct eigenfunctions from the $M$-function quotients,
\begin{alignat}{3}
    \varphi_1(x,y) &= -&&\frac{M_2}{M_1} = \frac{1-x}{x}, \quad &&\lambda_1 = 1 \label{eq:nonlinear_ex3_g1}\\
    \varphi_2(x,y) &= &&\frac{M_1}{M_3} = \frac{x}{y}, \quad &&\lambda_2 = 1.
\end{alignat}
The corresponding weight vectors are $\b{p}_1 = [-1 \ \ 1 \ \ 0]$ and $\b{p}_2 = [1 \ \ 0 \ -1]$. $\b{p}_1 $ and $\b{p}_2$ are linearly independent, meaning we may use $\varphi_1$ and $\varphi_2$ to solve for $\b{x}(t)$ according to Theorem~\ref{thm:lin_independence}.
Solving for $x$ and $y$ as functions of $\varphi_1$ and $\varphi_2$ and substituting the solutions for $\varphi_1(t)$ and $\varphi_2(t)$ gives us
\begin{align}
    x(t) &= \frac{1}{1+\varphi_1(t)} = \frac{1}{1+\frac{1-x_0}{x_0}e^t}\\
    y(t) &= \frac{1}{\varphi_2(t)(1+\varphi_1(t))} = \frac{y_0}{x_0e^t + (1-x_0)e^{2t}}.
\end{align}
Notice that the dynamics of $x$ and the solution for $x$ does not depend on $y$, making it a 1-dimensional ODE. Its eigenfunction, Eq.~\ref{eq:nonlinear_ex3_g1}, has a zero at $x=1$ and a singularity at $x=0$. Now when $x$ is viewed as the first variable in a 2-dimensional ODE, the first eigenfunction is now zero not only at a single point, $x=1$, but along the entire curve $x=1$ in the $(x,y)$ plane. Likewise its singularity now no longer exists only at a single point, $x=0$, but along the entire curve $x=0$ in the $(x,y)$ plane. This example demonstrates how zero and singular level-sets of eigenfunctions in $\mathds{R}^2$ are natural extensions of zero and singular values of eigenfunctions in $\mathds{R}$.

\subsection{Nonlinear example 4 --- quadratic invariant manifolds}

\begin{figure}[t]
    \centering
    \includegraphics[width=0.9\linewidth]{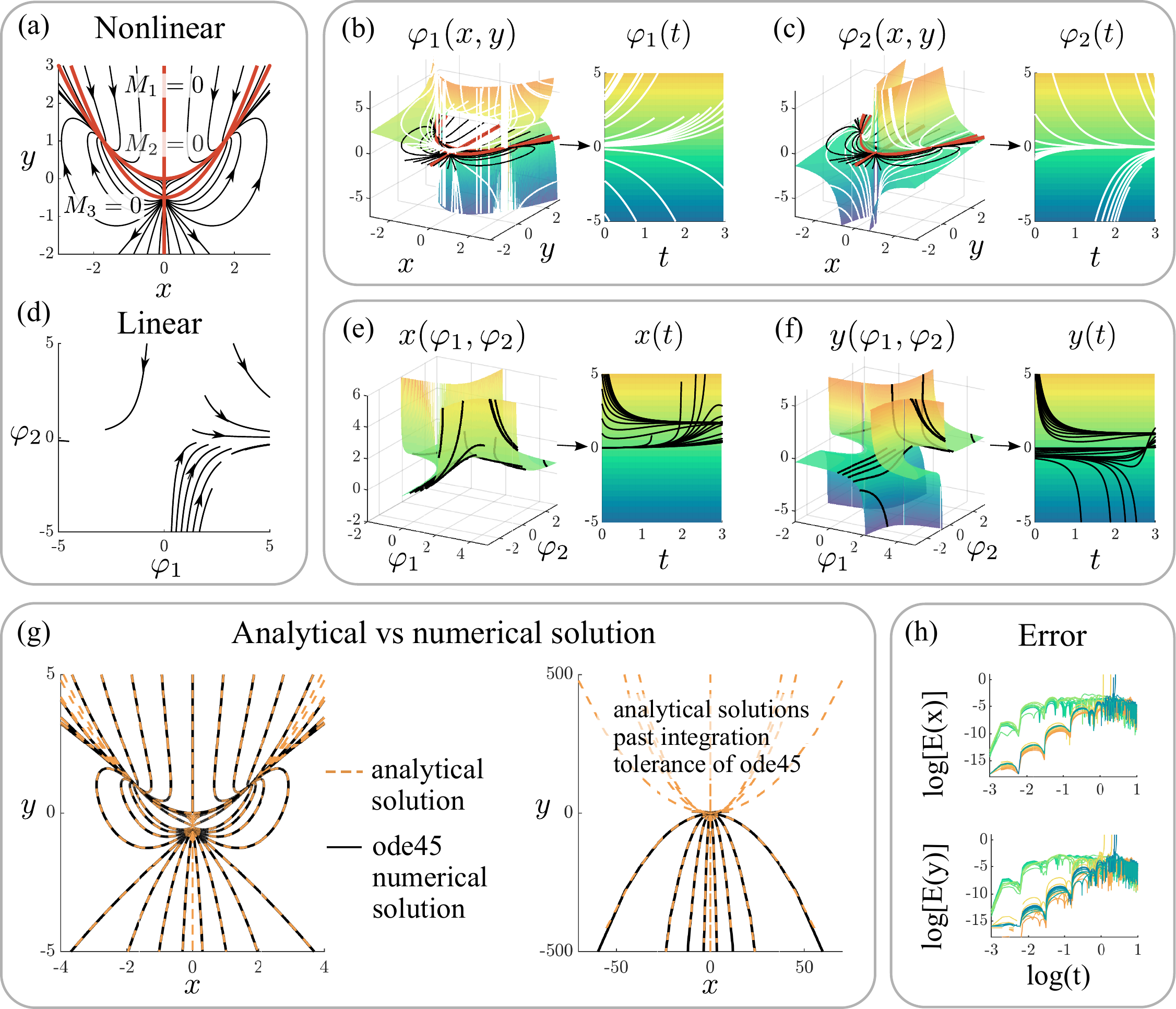}
    \caption{Dynamics in original space and eigenfunction space (a) Phase plane for Eq.~\ref{eq:nonlinear_ex4} (b-c) Flow lines projected onto $\varphi_1$ and $\varphi_2$ (d) Phase plane of ($\varphi_1,\varphi_2$) (e-f) Flow lines of ($\varphi_1, \varphi_2$) mapped back to original variables ($x,y$) (g) Analytical solution compared to numerical solution produced using ode45 (h) Error in trajectories over time}
    \label{fig:nonlinear_ex4}
\end{figure}

The examples we have considered thus far all have linear invariant manifolds. We now consider a system with quadratic manifolds intersecting the system's fixed points,
\begin{align} \label{eq:nonlinear_ex4}
\begin{split}
    \dot{x} &= x-xy\\
    \dot{y} &= -y+x^2-2y^2
\end{split}
\end{align}
with initial conditions $x(0) = x_0$ and $y(0) = y_0$. The invariant manifold generating functions for Eq.~\ref{eq:nonlinear_ex4} are $M_1 = x$, $M_2 = x^2-3y$, and $M_3 = 1-x^2+2y$; the invariant manifolds that go through the system's fixed points are $\Lambda_1 = \{(x,y): M_1(x,y)=0  \}$, $\Lambda_2 = \{(x,y): M_1(x,y)=0 \}$, and $\Lambda_3 = \{(x,y): M_1(x,y)=0 \}$. We check that these functions satisfy Eq.~\ref{eq:invar_man_req} and solve for the $N$-functions.
\begin{align*}
  \frac{d}{dt}M_1(x,y) &= \frac{dx}{dt} = x(1-y) = M_1 N_1\\
  \frac{d}{dt}M_2(x,y) &= \frac{d(x^2-3y)}{dt} =(x^2-3y)(-1-2y) = M_2 N_2\\
  \frac{d}{dt}M_3(x,y) &= \frac{d(1-x^2+2y)}{dt} = (1-x^2+2y)(-2y) = M_3 N_3
\end{align*}
The $N$-functions corresponding to the $M$-functions are $N_1(x,y) = 1-y$, $N_2(x,y) = -1-2y$, and $N_3(x,y) = -2y$. The following differences in the $N$-functions result in constants:
\begin{align*}
    \lambda_1 = &N_3(x,y) - N_2(x,y) = -2y - (-1-2y) = 1\\
    \lambda_2 = &N_3(x,y) -2 N_1(x,y) = -2y-(1-y)-(1-y) = -2.
\end{align*}
According to Theorem~\ref{thm:manifolds_general}, we can construct the following eigenfunctions from the $M$-functions:
\begin{alignat*}{2}
    \varphi_1(x,y) &= \frac{-3M_3}{M_2}= \frac{-3(1-x^2+2y)}{x^2-3y}, \quad &&\lambda_1 = 1\\
    \varphi_2(x,y) &=\frac{M_3}{M_1^2}= \frac{1-x^2+2y}{x^2}, \quad &&\lambda_2 = -2.
\end{alignat*}
The weight vectors for $\varphi_1$ and $\varphi_2$, $\b{p}_1 = [0 \ -1 \ \ 1]^T$ and $\b{p}_2 = [2 \ \ 0 \ \ 1]^T$, are linearly independent, therefore $(\varphi_1, \varphi_2)$ can be used to solve for $\b{x}(t)$.
Solving for $x$ and $y$ gives us two solutions this time,
\begin{align}
    x &= \frac{\pm \sqrt{3 \varphi_1}}{\sqrt{\varphi_1-6\varphi_2 + 3 \varphi_1 \varphi_2}}\\
    y &= \frac{\varphi_1 + 3\varphi_2}{\varphi_1 - 6\varphi_2 + 3 \varphi_1 \varphi_2}.
\end{align}
The correct sign for $x$ is resolved by taking the sign of the initial condition, $x_0$, resulting in the solution
\begin{align}
    x(t) &= \frac{\sign{(x_0)} \sqrt{ \frac{9(-1+x_0^2-2y_0)}{x_0^2-3y_0}e^t}}{\sqrt{\frac{3(-1+x_0^2-2y_0)}{x_0^2-3y_0}e^t-\frac{6(1-x_0^2+2y_0)}{x_0^2}e^{-2t} +  \frac{9(-1+x_0^2-2y_0)(1-x_0^2+2y_0)}{(x_0^2-3y_0)x_0^2} e^{-t}}}\\
    y(t) &= \frac{\frac{3(-1+x_0^2-2y_0)}{x_0^2-3y_0}e^t + \frac{3(1-x_0^2+2y_0)}{x_0^2}e^{-2t}}{\frac{3(-1+x_0^2-2y_0)}{x_0^2-3y_0}e^t-\frac{6(1-x_0^2+2y_0)}{x_0^2}e^{-2t} +  \frac{9(-1+x_0^2-2y_0)(1-x_0^2+2y_0)}{(x_0^2-3y_0)x_0^2} e^{-t}}.
\end{align}
Figure~\ref{fig:nonlinear_ex4}(a) shows the flow lines of Eq.~\ref{eq:nonlinear_ex4} and the system's quadratic invariant manifolds. The nonlinear trajectories are projected onto the eigenfunction space, resulting in linear dynamics (Fig.~\ref{fig:nonlinear_ex4}(b-c)). The linear system can be solved and the trajectories projected back to the original space (Fig.~\ref{fig:nonlinear_ex4}(d-f)). The numerical solution matches the analytical solution initially with growing error over time (Fig.~\ref{fig:nonlinear_ex4}(g-h)).

\subsection{Nonlinear example 5}

Consider the system with three linear invariant manifolds,
\begin{align}
\begin{split}
    \dot{x} &= x-y-x^2\\
    \dot{y} &= -x-y-xy.
\end{split}
\end{align}
The invariant manifolds going through the fixed points occur along $y = (1+\sqrt{2})x$, $y = (1-\sqrt{2})x$, and $y = x-2$.  The invariant manifold generating functions are $M_1 = y- (1+\sqrt{2})x$, $M_2 = y-(1-\sqrt{2})x$, and $M_3 = y-x+2$; the invariant manifolds are $\Lambda_1 = \{(x,y): M_1(x,y)=0  \}$, $\Lambda_2 = \{(x,y): M_1(x,y)=0 \}$, and $\Lambda_3 = \{(x,y): M_1(x,y)=0 \}$. We use Eq.~\ref{eq:invar_man_req} to solve for the corresponding $N$-functions.
\begin{align*}
  \frac{d}{dt}M_1(x,y) &= (y-(1+\sqrt{2})x)(\sqrt{2}-x) = M_1 N_1\\
  \frac{d}{dt}M_2(x,y) &= (y-(1-\sqrt{2})x)(-\sqrt{2}-x) = M_2 N_2\\
  \frac{d}{dt}M_3(x,y) &= (y-x+2)(-x) = M_3 N_3
\end{align*}
The $N$-functions are $N_1(x,y) = \sqrt{2}-x$, $N_2(x,y) = -\sqrt{2}-x$, and $N_3(x,y) = -x$. The following linear combinations of $N$-functions result in constants: $\lambda_1 = N_3(x,y) - N_2(x,y) = \sqrt{2}$, $\lambda_2  = N_3(x,y) - N_1(x,y) = -\sqrt{2}$, and $\lambda_3 = N_1(x,y)-N_2(x,y) = 2 \sqrt{2}$.
We construct the eigenfunctions 
\begin{align}
    \varphi_1(x,y) &= \frac{M_3}{2 M_2} = \frac{y-x+2}{2(y-(1-\sqrt{2})x)}, \quad \lambda_1 = \sqrt{2}\\
    \varphi_2(x,y) &= \frac{M_3}{2 M_1} = \frac{y-x+2}{2(y-(1+\sqrt{2})x)}, \quad \lambda_2 = -\sqrt{2}.
\end{align}
The weight vectors $\b{p}_1= [0 \ -1 \ \ 1]^T$ and $\b{p}_2= [-1 \ \ 0 \ \ 1]^T$ are linearly independent, meaning that we can solve for $\b{x}(t)$ using ($\varphi_1, \varphi_2$).
Solving the nonlinear system of equations $(\varphi_1, \varphi_2)$ for $(x,y)$ gives us
\begin{align}
    x &= \frac{-\sqrt{2}(\varphi_1-\varphi_2)}{-\varphi_1-\varphi_2+4\varphi_1\varphi_2}\\
    y &= \frac{(2-\sqrt{2})\varphi_1 + (2+\sqrt{2})\varphi_2}{-\varphi_1-\varphi_2+4\varphi_1 \varphi_2}.
\end{align}
Substituting the analytical solutions for $\varphi_1$ and $\varphi_2$ and the initial condition ($x_0, y_0$) gives us
\begin{align*}
    x(t) &= \frac{-\sqrt{2}\left(\frac{y_0-x_0+2}{2(y_0-(1-\sqrt{2})x_0)}e^{\sqrt{2}t}-\frac{y_0-x_0+2}{2(y_0-(1+\sqrt{2})x_0)}e^{-\sqrt{2}t} \right)}{-\frac{y_0-x_0+2}{2(y_0-(1-\sqrt{2})x_0)}e^{\sqrt{2}t}- \frac{y_0-x_0+2}{2(y_0-(1+\sqrt{2})x_0)}e^{-\sqrt{2}t} + \frac{(y_0-x_0+2)^2}{(y_0-(1-\sqrt{2})x_0)(y_0-(1+\sqrt{2})x_0)}}\\
    y(t) &= \frac{(2-\sqrt{2})\frac{y_0-x_0+2}{2(y_0-(1-\sqrt{2})x_0)}e^{\sqrt{2}t} + (2+\sqrt{2})\frac{y_0-x_0+2}{2(y_0-(1+\sqrt{2})x_0)}e^{-\sqrt{2}t}}{-\frac{y_0-x_0+2}{2(y_0-(1-\sqrt{2})x_0)}e^{\sqrt{2}t}- \frac{y_0-x_0+2}{2(y_0-(1+\sqrt{2})x_0)}e^{-\sqrt{2}t} + \frac{(y_0-x_0+2)^2}{(y_0-(1-\sqrt{2})x_0)(y_0-(1+\sqrt{2})x_0)}}.
\end{align*}

\subsection{Eigenfunctions from invariant manifolds in previous work}

Eigenfunctions constructed from a system's invariant manifolds can be found in previous work \cite{brunton_koopman_2016, bollt_matching_2018}.
In Ref.~\cite{brunton_koopman_2016} the nonlinear ODE
\begin{align}
\begin{split}
    \dot{x} &= -0.05 x\\
    \dot{y} &= -(y-x^2)
\end{split}
\end{align}
has eigenfunctions $\varphi_1 = M_1 = x$ and $\varphi_2 = M_2 = y-\frac{10}{9}x^2$ with corresponding eigenvalues $\lambda_1 = N_1 = -0.05$ and $\lambda_2 = N_2 = -1$. Both of these eigenfunctions are invariant manifold generating functions with zero level-sets that go through the fixed point at the origin. Eigenfunctions could be constructed from individual $M$-functions because the corresponding $N$-functions are constants, making them eigenvalues.

In Ref.~\cite{bollt_matching_2018} the nonlinear ODE
\begin{align}
\begin{split}
    \frac{dx}{dt} &= -2 y(x^2 - y - 2xy^2 + y^4) +  (x + 4 x^2 y - y^2 - 8 x y^3 + 4 y^5)\\
    \frac{dy}{dt} &= 2  (x - y^2)^2 - (x^2 - y - 2 x y^2 + y^4)
\end{split} \label{eq:2D_ex7_dxdt}
\end{align}
was found to have eigenfunctions $\varphi_1 = M_1 = x - y^2$ and $\varphi_2 = M_2 = -x^2 +y + 2xy^2 - y^4$ with corresponding eigenvalues $\lambda_1 = N_1 = 1$ and $\lambda_2 = N_2 =1$. Once again, these eigenfunctions are $M$-functions whose zero level-sets describe invariant manifolds that go through the system's fixed points. The corresponding $N$-functions are constants, once again allowing the $M$-functions to be eigenfunctions.

\section{Discussion}\label{sec:discussion}

We present a novel method for solving nonlinear ordinary differential equations. The crux of our method relies on finding explicit, globally-valid expressions for eigenfunctions by composing invariant manifold generating functions in the manner outlined in Theorems~\ref{thm:manifolds_general} and \ref{thm:manifolds_general2}. 
We demonstrate our method by finding analytical solutions for 2-dimensional nonlinear ODEs that were previously analytically intractable.
Currently, this method can only be used to solve nonlinear autonomous ODEs that have real fixed points and real invariant manifolds. Furthermore, according to Theorem~\ref{thm:manifolds_general2}, the variables in the $N$-functions must be cancelable, producing a constant.
These conditions are restrictive and few nonlinear ODEs meet this criteria. However, for those that do, we have shown that a method for obtaining an analytical solution exists.

\subsection{Complex-valued eigenfunctions}

In the 1-dimensional case we saw that mapping an ODE in $\mathds{R}$ to a space with linear dynamics sometimes requires the eigenfunction to be complex-valued, $\varphi: \mathds{R} \rightarrow \mathds{C}$. This is unequivocally the case when the ODE contains complex fixed points, Eqs.~\ref{eq:oneD_ex22_dxdt}, \ref{eq:oneD_ex3_dxdt}, but also can occur when the ODE contains only real fixed points, Eq.~\ref{eq:oneD_ex2_dxdt}. Just as in the 1-dimensional case, we expect that for 2-dimensional nonlinear ODEs, complex-valued eigenfunctions will sometimes be necessary to obtain linear dynamics, $\varphi: \mathds{R}^2 \rightarrow \mathds{C}$. 
Eigenfunctions may be able to be constructed from complex-valued invariant manifolds and manifolds eminating from complex-valued fixed points in 2-dimensional ODEs.
We already have an example of this in the linear case. Linear systems in $\mathds{R}^2$ that are spiral sinks or sources do not have real invariant manifolds, and yet these systems can be solved using complex-valued eigenvalues and eigenfunctions, $\varphi: \mathds{R}^2 \rightarrow \mathds{C}$, constructed from complex-valued invariant manifolds. We aim to investigate how to find and construct complex-valued invariant manifolds as well as how they may be composed to produce eigenfunctions.

\subsection{Extension to higher dimensions}

For the examples given in this work we did not consider systems higher than two dimensions. However, Theorems~\ref{thm:manifolds_general} and \ref{thm:manifolds_general2} are generally applicable to nonlinear ODEs in $\mathds{R}^n$ and so can be used to construct eigenfunctions for nonlinear systems of any dimension. Finding invariant manifolds for ODEs with more than two variables may prove to be more difficult than finding invariant manifolds for 2-dimensional ODEs.  Nonetheless, if $M$-functions can be found for an ODE in $\mathds{R}^n$ and enough independent eigenfunctions constructed from these $M$-functions, then the $n$-dimensional ODE may be solved using a similar procedure to that outlined in Algorithm~\ref{alg:koopman_method}.

\subsection{Ramifications for data-driven approximations of eigenfunctions}

Polynomials or other smooth continuous functions are often used as basis functions for the data-driven discovery of eigenfunctions \cite{williams_data-driven_2015,kutz_dynamic_2016, kaiser_data-driven_2021, brunton_koopman_2016, lusch_deep_2018, li_extended_2017}. Choosing smooth and continuous basis functions for data-driven discovery assumes the resulting eigenfunctions, linearly composed from the basis functions, are smooth and continuous.
Many of the eigenfunctions found in this work are not continuous; it is impossible to generate continuous eigenfunctions that are globally valid for some ODEs. In these instances, forcing the approximate eigenfunctions to be continuous by using smooth continuous basis functions would result in a poor approximation, particularly around discontinuities.

Deriving globally valid eigenfunctions or exact expressions may not be necessary depending on the desired use for the eigenfunctions.
Nonlinear dynamics can be linearized for certain regions around fixed points with continuous approximate eigenfunctions \cite{lan_linearization_2013, mauroy_global_2016, page_koopman_2019, lusch_deep_2018}; these methods, in essence, create approximate local eigenfunctions for different regions of the domain that extend farther and are more accurate than the simple approximate eigenfunctions obtained through linearization using the Jacobian \cite{strogatz_nonlinear_2016, brunton_koopman_2016, kaiser_data-driven_2021}. DMD methods are often used to approximate eigenfunctions and produce different subspaces of observables for the region around each fixed point in the dynamical system \cite{kutz_dynamic_2016, page_koopman_2019, kaiser_data-driven_2021, bollt_matching_2018}. Our method, in contrast, produces a single set of observables for the entire space. If approximate solutions and control are desired only for a region around a particular fixed point, then using smooth functions to approximate eigenfunctions for a particular region may be suitable.
Otherwise, data-driven discovery of eigenfunctions that allow for the discovery of rational eigenfunctions and eigenfunctions that have discontinuities may be useful in finding globally valid eigenfunctions and more exact and simpler mappings between spaces with nonlinear and linear dynamics.

\subsection{Comparison to the method of characteristics solution for eigenfunctions}

The eigenfunctions of a nonlinear ODE, Eq.~\ref{eq:general_ode}, are solutions to Eq.~\ref{eq:lin_dyn_varphi_pde}, a linear PDE which is solvable via the method of characteristics \cite{bollt_matching_2018, bollt_geometric_2021, McOwenRobertC2003}. 
The method of characteristics propagates an initial condition curve, $\Gamma$, forward in time creating an integral surface that is a solution to the PDE \cite{McOwenRobertC2003}. In general, the method of characteristics does not produce simple, closed-form solutions.
The $\Gamma$ chosen to be the initial condition curve must traverse the flow. As there is an infinite number of curves that can be chosen for the initial condition, there is an infinite number of resulting eigenfunctions \cite{bollt_matching_2018}. Most of the transverse flows chosen to be initial conditions will result in eigenfunctions that do not have simple expressions. 
The eigenfunctions with simple analytical expressions found in this manuscript correspond to a method of characteristics solution for a specially chosen initial conditions curve $\Gamma$. We have not solved for these curves, and it remains an open question as to how to select $\Gamma$ such that the resulting eigenfunction will have a simple closed-form expression. This line of questioning is not addressed in this paper and is an avenue for further investigation.

\section{Conclusion}\label{sec:conclusion}

We solve nonlinear ordinary differential equations by finding Koopman eigenfunctions, from which we can construct solutions. Koopman eigenfunctions are generally difficult to find; we outline how Koopman eigenfunctions can be constructed from a system's invariant manifolds when certain conditions are met. We provide formulas for sets of eigenvalue-eigenfunction pairs which can then be used to solve the nonlinear system. We demonstrate the Koopman eigenfunction method of solving nonlinear ODEs on 1-dimensional and 2-dimensional ODEs. We use this method to solve nonlinear ODEs in $\mathds{R}^2$ that thus far had no known analytical solution. We suggest ways this method may be extended to solve a larger class of nonlinear ODEs and compare our method to methods for data-driven discovery of eigenfunctions as well as the method of characteristics, used to solve linear PDEs. We hope to make this method generalizable to solve a larger class of ODEs by improving methods of finding invariant manifolds and constructing eigenfunctions.

\subsection*{Acknowledgments}

MM acknowledges support from the National Science Foundation Mathematical Sciences Postdoctoral Research Fellowship (award no. 2103239).
JNK acknowledges funding support from the Air Force Office of Scientific Research (AFOSR FA9550- 19-1-0386) and from the National Science Foundation AI Institute in Dynamic Systems grant number 2112085. 
We would like to thank Marko Budišić for his valuable comments on our manuscript.



\end{document}